\documentclass[10pt]{article}
\usepackage[english]{babel}
\usepackage{amsmath}
\usepackage{amsthm}
\usepackage{amssymb}
\usepackage[notref,notcite]{showkeys}

\newtheorem{theorem}{Theorem}
\newtheorem{proposition}[theorem]{Proposition}
\newtheorem{example}[theorem]{Example}
\newtheorem{definition}[theorem]{Definition}

\newtheorem{lemma}[theorem]{Lemma}
\newtheorem{remark}[theorem]{Remark}

\newcommand{\mL}{{\cal L}}

\newcommand{\R}{\mathbb{R}}

\newcommand{\beq}{\begin{equation}}
\newcommand{\eeq}{\end{equation}}

\newcommand{\C}{\mathbb{C}}
\newcommand{\N}{\mathbb{N}}
\newcommand{\Z}{{\mathbb{Z}}}
\newcommand{\A}{{\cal A}}

\newcommand{\mD}{{\mathcal{D}}}

\newcommand{\mX}{{\mathcal X}}
\newcommand{\mA}{{\mathcal A}}

\begin{document}

\title{\bf Fractional derivative in spaces of continuous functions}

\author{Davide Guidetti}

\maketitle

Let $X$ be a Banach space, let $T \in \R^+$. We consider the Banach space $C([0, T]; X)$ of continuous functions with values in $X$. 
We introduce the following operator $B$: 

\begin{equation}\label{eq1}
\left\{\begin{array}{l}
D(B):= \{u \in C^1(0, T; X) : u(0) = 0\}, \\ \\
Bu(t) = u'(t). 
\end{array}
\right. 
\end{equation}
Then it is easy to see that $\rho(B) = \C$. Moreover, $\forall \lambda \in \C$, $\forall f \in C([0, T]; X)$, 
\begin{equation}\label{eq2}
(\lambda - B)^{-1}f(t) = - \int_0^t e^{\lambda (t-s)} f(s) ds. 
\end{equation}
One has 
\begin{equation}\label{eq3}
\|(\lambda - B)^{-1}\|_{\mL(C([0, T]; X)} \leq \frac{1 - e^{Re(\lambda)T}}{Re(\lambda)}. 
\end{equation}
The estimate is extensible to the case $\lambda = 0$, as $\| B^{-1}\|_{\mL(C([0, T]; X))} \leq T$. 

\begin{definition} \label{de1} Let $B$ be a linear   operator in  the Banach space $Y$ (that is, $B : D(B)  \to Y$, with $D(B)$ linear subspace of $Y$). We shall say that it is of type $\omega$, with $0 < \omega < \pi$, if $\rho(B)$ contains 
$\{\lambda : |Arg(\lambda)| > \omega\}$  and, if $\epsilon > 0$, there exists $M(\epsilon) > 0$ such that $\|\lambda (\lambda - A)^{-1}\|_{\mL(Y)} \leq M(\epsilon)$ in case $|Arg(\lambda)| \geq \omega + \epsilon$.
\end{definition}

So  $B$ is  class $\frac{\pi}{2}$. Moreover, $0 \in \rho(B)$. Following \cite{Ta1} (who considers only the case $D(B)$ dense in $Y$), we can define, $\forall \alpha \in \R^+$, the operator $B^{-\alpha}$: we fix $a > 0$, $\theta \in (\frac{\pi}{2}, \pi)$ and the contour 
$$
\Gamma := \{\lambda \in \C : |Arg(\lambda - a)| = \theta\},
$$
oriented from $\infty e^{-i\theta}$ to $\infty e^{-i\theta}$ and set
\begin{equation}\label{eq4}
B^{-\alpha} := - \frac{1}{2\pi i} \int_\Gamma \lambda^{-\alpha} (\lambda - B)^{-1} d\lambda. 
\end{equation}
It is well known that $\forall \alpha, \beta \in \R^+$, 
$$B^{-\alpha} B^{-\beta} = B^{-(\alpha + \beta)}.$$
Then, if $f \in C([0, T]; X)$, we have
\begin{equation}
\begin{array}{c}
B^{-\alpha} f(t) = \frac{1}{2\pi i} \int_\Gamma \lambda^{-\alpha} (\int_0^t e^{\lambda(t-s)} f(s) ds ) d\lambda \\ \\
= \frac{1}{2\pi i} \int_0^t (\int_\Gamma \lambda^{-\alpha} e^{\lambda(t-s)} d\lambda) f(s) ds
\end{array}
\end{equation}

\begin{proposition}\label{pr2}
If $\alpha \in \R^+$ and $f \in C([0, T]; X)$, 
$$
B^{-\alpha} f(t) = \frac{1}{\Gamma (\alpha)} \int_0^t (t - s)^{\alpha - 1} f(s) ds.
$$
\end{proposition}

\begin{proof}
The formula is true if $\alpha = 1$. If $\alpha = n \in \N$, 
$$
B^{-n}f(t) = \frac{1}{(n-1)!} \int_0^t (t-s)^{n-1} f(s) ds = \frac{1}{\Gamma(n)} \int_0^t (t-s)^{n-1} f(s) ds. 
$$
We consider the case $\alpha \in (0, 1)$: we have
$$
B^{-\alpha}f(t)= \frac{1}{2\pi i} \int_0^t (\int_\Gamma \lambda^{-\alpha} e^{\lambda(t-s)} d\lambda) f(s) ds
$$
By Cauchy's theorem, if $r \in \R^+$, 
$$
\begin{array}{c}
 \frac{1}{2\pi i} \int_\Gamma \lambda^{-\alpha} e^{\lambda r} d\lambda = \lim_{\theta \to 0^+} \frac{1}{2\pi i} (\int_0^\infty (\rho e^{i\theta})^{-\alpha} e^{r\rho e^{i\theta}} e^{i\theta} d\rho - \int_0^\infty (\rho e^{-i\theta})^{-\alpha} e^{r\rho e^{-i\theta}} e^{-i\theta} d\rho )  \\ \\
 = \frac{1}{2\pi i} (- \int_0^\infty \rho^{-\alpha} e^{-i \alpha \pi}  e^{-r\rho} d\rho + \int_0^\infty \rho^{-\alpha} e^{i \alpha \pi}  e^{-r\rho} d\rho) \\ \\
 = \frac{\sin(\alpha \pi)}{\pi} \int_0^\infty e^{-r\rho} \rho^{-\alpha} d\rho = \frac{\sin(\alpha \pi) \Gamma (1-\alpha)}{\pi} r^{\alpha -1} 
 = \frac{r^{\alpha -1}}{\Gamma(\alpha)},
 \end{array}
$$
employing the well known formula
$$
\Gamma (z) \Gamma (1-z) = \frac{\pi}{\sin(\pi z)}.
$$
So the case $\alpha \in (0, 1)$ is proved. 

Finally, let $\alpha = n + \beta$, with $n \in \N$, $\beta \in (0, 1)$. We have
$$
\begin{array}{c}
B^{-\alpha}f(t) = B^{-\beta}(B^{-n}f)(t) = \frac{1}{\Gamma (n) \Gamma (\beta)} \int_0^t (t - \tau)^{\beta -1} (\int_0^\tau (\tau - s)^{n-1} f(s) ds) d\tau \\ \\
= \int_0^t \frac{1}{\Gamma (n) \Gamma (\beta)} (\int_s^t (t-\tau)^{\beta - 1} (\tau - s)^{n-1} d\tau ) f(s) ds
\end{array} 
$$
We have 
$$
\frac{1}{\Gamma (n) \Gamma (\beta)} \int_s^t (t-\tau)^{\beta - 1} (\tau - s)^{n-1} d\tau = \frac{B(\beta,n)}{\Gamma (n) \Gamma (\beta)} (t-s)^{n+\beta -1},
$$
where $B$ indicates Euler's Beta function. From the classical formula
$$
B(z,w) = \frac{\Gamma (z) \Gamma (w)}{\Gamma (z+w)}
$$
the conclusion follows. 

\end{proof}

It can be seen that, $\forall \alpha \in \R^+$, $B^{-\alpha}$ is injective. So we can define 
$$
B^\alpha := (B^{-\alpha})^{-1}. 
$$
By Proposition 2.3.1 in \cite{Ta1}, $B^\alpha$ is a closed operator, if $0 < \alpha < \beta$ $D(B^\beta) \subseteq D(B^\alpha)$ and $B^\alpha B^\beta = B^{\alpha + \beta}$.

\begin{remark}
{\rm Given $f \in C([0, T]; X)$, if $m \in \N_0$ and $m < \alpha < m+1$,  the Riemann-Liouville derivative of order $\alpha$ is usually defined (see, for example, \cite{Po1}, Chapter 2.3) as
\begin{equation}\label{eq6}
g:= D_t^{m+1} (\frac{1}{\Gamma(\alpha - m)} \int_0^t (t-s)^{m - \alpha} f(s) ds) = D_t^{m+1} (B^{\alpha - m-1}f)(t),
\end{equation}
in  case $B^{\alpha - m-1}f \in C^{m+1}([0, T]; X)$. If $0 < \alpha < 1$, the operator $B^\alpha$ coincides with the Riemann-Liouville derivative of order $\alpha$. In fact, let $B^{\alpha -1}f \in 
C^1([0, T]; X)$. As $B^{\alpha -1}f(0) = 0$, $B^{\alpha -1}f \in D(B)$. So

 $$g = BB^{\alpha -1}f, B^{\alpha -1}f = B^{-1}g, 
f = B^{1-\alpha} B^{-1}g = B^{1-\alpha}B^{\alpha-1}B^{-\alpha}g = B^{-\alpha}g.$$
We deduce $f \in D(B^\alpha)$ and
$$
B^\alpha f = g. 
$$

}
\end{remark}
If $\alpha, \beta \in \R^+$, we introduce the following function of Mittag-Leffler type, which is an entire (see \cite{Po1}, Chapter 1.2.1): 
$$
E_{\alpha,\beta}(z) = \sum_{k=0}^\infty \frac{z^k}{\Gamma(\alpha k + \beta)}. 
$$

\begin{proposition}
Let $\alpha \in \R^+$, Then $\rho(B^\alpha) = \C$. Moreover, $\forall \lambda \in \C$, $\forall t \in [0, T]$, 
\begin{equation}\label{eq7}
[(\lambda - B^\alpha)^{-1} f](t) = - \int_0^t E_{\alpha,\alpha} (\lambda(t-s)^\alpha) (t-s)^{\alpha -1} f(s) ds. 
\end{equation}
\end{proposition}

\begin{proof}
We already know that $0 \in \rho(B^\alpha)$ and (\ref{eq7}) is true if $\lambda = 0$. Suppose $\lambda \neq 0$. Then the equation $(\lambda - B^\alpha)u = f$ ($f \in C([0, T]; X)$) is equivalent to
\begin{equation}\label{eq8}
u = \lambda B^{-\alpha}u - B^{-\alpha}f. 
\end{equation}
The linear operator $\lambda B^{-\alpha}$ has spectral radius $0$. In fact, 
$$
\|(\lambda B^{-\alpha})^k\|_{\mL(C([0, T]; X))}^{1/k} = |\lambda| \| B^{-\alpha k}\|_{\mL(C([0, T]; X))}^{1/k} \leq \frac{|\lambda| T^\alpha }{\Gamma(\alpha k + 1)^{1/k}} \to 0 \quad (k \to \infty),
$$
because the power series $\sum_{k=0}^\infty \frac{z^k}{\Gamma(\alpha k + 1)}$ has convergence radius $\infty$. So equation (\ref{eq8}) has the unique solution
$$
u = - \sum_{k=0}^\infty \lambda^k B^{-\alpha k - \alpha} f = - \sum_{k=0}^\infty  \frac{\lambda^k}{\Gamma(\alpha k + \alpha)} \int_0^t (t-s)^{\alpha k + \alpha -1} f(s) ds = - \int_0^t E_{\alpha,\alpha}(\lambda (t-s)^\alpha)) (t-s)^{\alpha -1} f(s) ds. 
$$
\end{proof}

Let  $\theta \in (0, \pi)$, $r \in \R^+$. We indicate with $\Gamma(\theta,r)$ a generic piecewise regular, simple contour describing 
\begin{equation}
\begin{array}{c}
\{\mu \in \C \setminus (-\infty, 0] : |Arg(\mu)| = \theta, |\mu| \geq r\} \cup \{\mu \in \C \setminus (-\infty, 0] : |Arg(\mu)| \leq \theta, |\mu| = r\}.
\end{array}
\end{equation}

\begin{proposition}\label{pr5} Let $\alpha \in (0, 2)$. Then:

(I) let $\lambda \in \C$. Let $\Gamma(\lambda) = \Gamma(\theta,r)$, for some $\theta \in (\frac{\pi}{2}, \pi)$ with $\alpha \theta < \pi$, $r^\alpha > |\lambda|$.  If $f \in C([0, T]; X)$, 
\begin{equation}\label{eq9}
(\lambda - B^\alpha)^{-1}f(t) = - \frac{1}{2\pi i} \int_{\Gamma(\lambda)} (\lambda - \mu^\alpha)^{-1} (\mu - B)^{-1} f d\mu = \frac{1}{2\pi i} 
 \int_0^t ( \int_{\Gamma(\lambda)} \frac{e^{\mu (t-s)}}{\lambda - \mu^\alpha} d\mu) f(s) ds. 
\end{equation}

(II) $B^\alpha$ is of type $\frac{\alpha \pi}{2}$. 
\end{proposition}

\begin{proof}
We set
$$
T(\lambda):= - \frac{1}{2\pi i} \int_{\Gamma(\lambda)} (\lambda - \mu^\alpha)^{-1} (\mu - B)^{-1} d\mu. 
$$
It is easily seen (observing that the integral does not depend on the specific choice of $\Gamma(\lambda)$ and can be chosen locally independently of $\lambda$) that $T$, with domain $\{\lambda \in \C \setminus \{0\} : |Arg(\lambda)| > \frac{\alpha \pi}{2}\}$  is entire with values in $\mL(C([0, T]; X)$. 

 By well known facts of analytic continuation, in order to show that $T(\lambda) = (\lambda - B^\alpha)^{-1}$, it is sufficient to show that this holds if $\lambda$ belongs to some ball centred in $0$. We set
$R(\lambda):= (\lambda - B^\alpha)^{-1}$. We prove that $T^{(k)}(0) = R^{(k)}(0)$ for every $k \in \N_0$. In fact, we have
$$
R^{(k)}(0) = - k! B^{-(k+1)\alpha}.
$$
On the other hand, taking $\Gamma(\lambda) = \Gamma$, 
$$
T^{(k)}(\lambda) = \frac{(-1)^{k+1} k!}{2\pi i} \int_\Gamma (\lambda - \mu^\alpha)^{-k-1} (\mu - B)^{-1} d\mu, 
$$
so that 
$$
T^{(k)}(0) = \frac{ k!}{2\pi i} \int_\Gamma  \mu^{-\alpha(k+1)} (\mu - B)^{-1} d\mu = - k! B^{-(k+1)\alpha},
$$
and the conclusion follows. 

We show that, if $\phi \in (\frac{\alpha \pi}{2}, \pi]$, there exists $C(\phi)$ in $\R^+$ such that, $\forall \rho \in \R^+$, 
$$
\|T(\rho e^{i\phi})\|_{\mL(C([0, T]; X)} \leq \frac{C(\phi)}{\rho}. 
$$
Let $\Gamma = \Gamma(e^{i\phi})$. Then we can take $\Gamma(\rho e^{i\phi}) = \rho^{1/\alpha} \Gamma$. So we have 
$$
\begin{array}{c}
\|T(\rho e^{i\phi})\|_{\mL(C([0, T]; X)} =  \| \frac{1}{2\pi i} \int_{ \rho^{1/\alpha} \Gamma} (\rho e^{i\theta} - \mu^\alpha)^{-1} (\mu - B)^{-1} d\mu\|_{\mL(C([0, T]; X))} \\ \\
= \| \frac{1}{2\pi i} \rho^{1/\alpha -1}  \int_{ \Gamma} (e^{i\theta} - \mu^\alpha)^{-1} (\rho^{1/\alpha} \mu - B)^{-1} d\mu \|_{\mL(C([0, T]; X))} \\ \\
\leq C_1 \rho^{-1} \int_{ \Gamma} |e^{i\theta} - \mu^\alpha|^{-1} |\mu|^{-1} d|\mu| = C(\theta) \rho^{-1}. 
\end{array}
$$
\end{proof}

\begin{remark}
{\rm Comparing (\ref{eq7}) with (\ref{eq9}), we deduce, for $\lambda \in \C$,  $t \in \R^+$, 
$$
   \frac{1}{2\pi i} \int_{\Gamma(\lambda)} \frac{e^{\mu t}}{\mu^\alpha - \lambda} d\mu =  E_{\alpha,\alpha}(\lambda t^\alpha) t^{\alpha - 1}. 
$$

}
\end{remark}

The following proposition formally justifies the use of the Laplace transforms with fractional derivatives:

\begin{proposition} \label{pr7A} Let $\theta_0 \in (\frac{\pi}{2}, \pi)$, $R \in \R^+$, $\alpha \in [0, \infty)$. Let $F : \{\lambda \in \C : |\lambda| > R, |Arg(\lambda)| < \theta_0\} \to X$ be such that: 

(a) $F$ is holomorphic; 

(b) there exists $M \in \R^+$ such that $\|F(\lambda)\| \leq M |\lambda|^{-1-\alpha}$, if $\lambda \in \C$,  $|\lambda| > R$, $|Arg(\lambda)| < \theta_0$; 

(c) for some $F_0 \in X$, ${\displaystyle \lim_{|\lambda| \to \infty}} \lambda^{1+\alpha} F(\lambda) = F_0$. 

Let $\gamma$ be a piecewise regular contour, describing, for some $\theta_1 \in (\frac{\pi}{2}, \theta_0)$, $R_1 > R$,
 $$\{\lambda \in \C : |\lambda| \geq R_1, |Arg(\lambda)| = \theta_1\} \cup \{\lambda \in \C : |\lambda| = R_1, |Arg(\lambda)| \leq \theta_1 \}, $$
 oriented from $\infty e^{-i \theta_1}$ to $\infty e^{i \theta_1}$. We set, for $t \in (0, T]$, 
 $$
 u(t) = \frac{1}{2\pi i} \int_\gamma e^{\lambda t} F(\lambda) d\lambda. 
 $$
 Then $u \in D(B^\alpha)$, for $t \in (0, T]$, 
 $$
 B^\alpha u(t) = \frac{1}{2\pi i} \int_\gamma e^{\lambda t} \lambda^{\alpha} F(\lambda) d\lambda
 $$
 and
 $$
 B^\alpha u(0) = F_0. 
 $$
\end{proposition}

\begin{proof}
We begin by considering the case $\alpha = 0$. It is clear that $u \in C((0, T]; X)$. We show that
$$
{\displaystyle \lim_{t \to 0}} u(t) = F_0. 
$$
By standard properties of holomorphic functions, we have, for $t \in (0, \min\{1, T\}]$, 
$$
\begin{array}{c}
u(t) = \frac{1}{2\pi i} \int_{t^{-1}\gamma} e^{\lambda t} F(\lambda) d\lambda 
= \frac{1}{2\pi i} \int_{\gamma} \frac{e^{\lambda}}{\lambda} t^{-1} \lambda F(t^{-1} \lambda) d\lambda \\ \\
= F_0 + \frac{1}{2\pi i} \int_{\gamma} \frac{e^{\lambda}}{\lambda} [ t^{-1} \lambda F(t^{-1} \lambda) - F_0] d\lambda
\end{array}
$$
and the second summand vanishes as $t \to 0$, by the dominated convergence theorem. 

Suppose now that $\alpha > 0$. We set, for $t \in (0, T]$, 
$$
f(t) = \frac{1}{2\pi i} \int_\gamma e^{\lambda t} \lambda^{\alpha} F(\lambda) d\lambda .
$$
Then, employing what we have seen in case $\alpha = 0$, we deduce $f \in C([0, T]; X)$ and $f(0) = F_0$. We check that
$$
B^{-\alpha}f(t)= u(t). 
$$
We fix $\Gamma$  piecewise regular contour, describing, for some $\theta_2 \in (\frac{\pi}{2}, \theta_1)$, $R_2 > R_1$,
 $$\{\lambda \in \C : |\lambda| \geq R_2, |Arg(\lambda)| = \theta_2\} \cup \{\lambda \in \C : |\lambda| = R_2,  |Arg(\lambda)| \leq \theta_2 \}, $$
 oriented from $\infty e^{-i \theta_2}$ to $\infty e^{i \theta_2}$. If $\lambda \in \C \setminus \gamma$, we have 
 $$
 (\lambda - B)^{-1}f(t) = \frac{1}{2\pi i} \int_\gamma e^{\mu t}(\lambda - \mu)^{-1} \mu^{\alpha} F(\mu) d\mu .
 $$
 So, by (\ref{eq4}), 
 $$
 \begin{array}{c}
 B^{-\alpha} f(t) = - \frac{1}{(2\pi i)^2} \int_\Gamma \lambda^{-\alpha} (\int_\gamma \frac{e^{\mu t}}{\lambda - \mu} \mu^{\alpha} F(\mu) d\mu) d\lambda \\ \\
 = - \frac{1}{(2\pi i)^2} \int_\gamma (\int_\Gamma \lambda^{-\alpha} (\lambda - \mu)^{-1} d\lambda) e^{\mu t} \mu^{\alpha} F(\mu) d\mu \\ \\
 = \frac{1}{2\pi i} \int_\gamma e^{\lambda t}  F(\lambda) d\lambda = u(t). 
 \end{array}
 $$

\end{proof}

\begin{remark}\label{re9}
{\rm From the inversion formula of the Laplace transform, we have, it $\delta \in \R^+$, $t \geq 0$, 
$$
t^\delta = \frac{1}{2\pi i} \int_\gamma e^{\lambda t} \Gamma(\delta + 1) \lambda^{-\delta - 1} d\lambda. 
$$
We deduce that, if $\delta \geq \alpha > 0$, $u(t) = t^\delta$ belongs to $D(B^\alpha)$ and
$$
B^\alpha u(t) = \frac{1}{2\pi i} \int_\gamma e^{\lambda t} \Gamma(\delta + 1) \lambda^{\alpha -\delta - 1} d\lambda = \frac{\Gamma (\delta +1)}{\Gamma (\delta + 1 -\alpha)} t^{\delta -\alpha}. 
$$ }
\end{remark}

We pass to consider the domain $D(B^\alpha)$ of $B^\alpha$, with $\alpha > 0$. We introduce the following function spaces, for $m \in \N_0$, $\beta \in (0, 1)$: 
\begin{equation}
B^{m+\beta}_{\infty,1}([0, T]; X) := \{f \in C^m([0, T]; X) : \int_0^T h^{-\beta-1} \|f^{(m)}(. + h) - f^{(m)}\|_{C([0, T-h]; X)} dh < \infty\}, 
\end{equation}
\begin{equation}
C^{m+\beta}([0, T]; X) := \{f \in C^m([0, T]; X) : sup_{0 < h < T} h^{-\beta} \|f^{(m)}(. + h) - f^{(m)}\|_{C([0, T-h]; X)} < \infty\}. 
\end{equation}
$C^{\beta}([0, T]; X)$ is the classical space of H\"older continuous functions, $B^{m+\beta}_{\infty,1}([0, T]; X)$ is a particular type of Besov space. It is known that, if $0 < \beta < 1$ and $\epsilon > 0$, 
$$
C^{\beta+\epsilon}([0, T]; X) \hookrightarrow B^{\beta}_{\infty,1}([0, T]; X) \hookrightarrow C^{\beta}([0, T]; X). 
$$
We set also, again for $m \in \N$, $\beta \in (0, 1)$:
$$
{\stackrel{o}{B}}^{m+\beta}_{\infty,1}([0, T]; X):= \{f \in B^{m+\beta}_{\infty,1}([0, T]; X) : f^{(k)}(0) = 0, k \in \N_0,  k \leq m\},
$$
for $m \in \N$, $\beta \in (0, 1]$, 
\begin{equation}
{\stackrel{o}{C}}^{m+\beta}([0, T]; X) := \{f \in C^{m+\beta}([0, T]; X): f^{(k)}(0) = 0, k \in \N_0, k \leq m\}. 
\end{equation}
Observe that
$$
{\stackrel{o}{C}}^{m+1}([0, T]; X) = D(B^{m+1}). 
$$

\begin{proposition}\label{pr7}
Let $\alpha \in \R^+ \setminus \N$. Then
$$
{\stackrel{o}{B}}^{\alpha}_{\infty,1}([0, T]; X)  \subseteq D(B^\alpha) \subseteq {\stackrel{o}{C}}^{\alpha}([0, T]; X).
$$
\end{proposition}

\begin{proof} If $m \in \N$ and $m < \alpha < m+1$, 
$$
\begin{array}{c}
D(B^\alpha) = \{u \in D(B^m): B^m u  \in D(B^{\alpha-m})\} \\ \\
= \{u \in C^m([0, T] ; X) : u^{(k)}(0) = 0, k = 0, \dots, m-1, u^{(m)} \in D(B^{\alpha - m})\}. 
\end{array}
$$
So we are reduced to consider the case $\alpha \in (0, 1)$. It is know that 
$$
(C([0, T]; X), D(B))_{\alpha,1} = {\stackrel{o}{B}}^{\alpha}_{\infty,1}([0, T]; X) = \{f \in B^{\alpha}_{\infty,1}([0, T]; X): f(0) = 0\},
$$
(see for this \cite{Gu1}), where we  indicate with $(.,.)_{\alpha,p}$ ($0 < \alpha < 1$, $p \in [1, \infty]$) the corresponding real interpolation functor. In order to show that 
$$(C([0, T]; X), D(B))_{\alpha,1} \hookrightarrow D(B^\alpha)$$
we can try to show that there exists $C > 0$, such that, $\forall u \in D(B)$, 
\begin{equation}\label{eq13A}
\|B^\alpha u\|_{C([0, T]; X)} \leq C \|u\|_{C([0, T]; X)}^{1-\alpha} \|Bu\|_{C([0, T]; X)}^{\alpha}. 
\end{equation}
(see for this \cite{Tr1}, Lemma 1.10.1). So let $u \in D(B) = {\stackrel{o}{C}}^{1}([0, T]; X)$. For every $\lambda \in \R^+$ we have
\begin{equation}\label{eq13}
B^\alpha u =  \lambda B^\alpha(\lambda + B)^{-1} u + B^\alpha(\lambda + B)^{-1} Bu. 
\end{equation}
We estimate $\|B^\alpha(\lambda + B)^{-1}\|_{\mL(C([0, T]; X))}$. We have
$$
B^\alpha(\lambda + B)^{-1} = B^{\alpha-1}B(\lambda + B)^{-1}  = B^{\alpha-1} - \lambda B^{\alpha-1}(\lambda + B)^{-1}). 
$$
We indicate with $\Gamma$ a piecewise regular path, connecting $\infty e^{-i\theta}$ to $\infty e^{i\theta}$, for some $\theta \in (\frac{\pi}{2}, \pi)$, contained in $\C \setminus (-\infty, 0]$ and with $\Gamma'$ the path $\Gamma + 1$. 
Then, by the theorem of Cauchy, we have 
$$
(\lambda + B)^{-1} = - \frac{1}{2\pi i} \int_{\Gamma'} (\lambda + \nu)^{-1} (\nu - B)^{-1} d\nu, 
$$
$$
\begin{array}{c}
B^{\alpha-1}(\lambda + B)^{-1} = \frac{1}{(2\pi i)^2} \int_\Gamma (\int_{\Gamma'} \mu^{\alpha -1} (\lambda + \nu)^{-1} (\mu - B)^{-1}(\nu - B)^{-1} d\nu) d\mu \\ \\
= \frac{1}{(2\pi i)^2} \int_\Gamma (\int_{\Gamma'} (\lambda + \nu)^{-1} (\nu - \mu)^{-1} d\nu) \mu^{\alpha - 1} (\mu - B)^{-1} d\mu \\ \\
- \frac{1}{(2\pi i)^2} \int_{\Gamma'} (\int_{\Gamma} \mu^{\alpha - 1} (\nu - \mu)^{-1} d\mu ) (\lambda + \nu)^{-1} (\nu - B)^{-1} d\nu,
\end{array}
$$
by the resolvent identity. 
We have
$$
\int_{\Gamma'} (\lambda + \nu)^{-1} (\nu - \mu)^{-1} d\nu = 0, 
$$
$$
\int_{\Gamma} \mu^{\alpha - 1} (\nu - \mu)^{-1} d\mu = 2\pi i \nu^{\alpha -1}. 
$$
So 
$$
\begin{array}{c}
B^\alpha(\lambda + B)^{-1} = - \frac{1}{2\pi i} \int_\Gamma (\mu^{\alpha - 1} - \frac{\lambda \mu^{\alpha -1}}{\lambda + \mu}) (\mu - B)^{-1} d\mu \\ \\
= - \frac{1}{2\pi i} \int_\Gamma  \frac{\mu^\alpha}{\lambda + \mu} (\mu - B)^{-1} d\mu = - \frac{1}{2\pi i} \int_{\lambda \Gamma}  \frac{\mu^\alpha}{\lambda + \mu} (\mu - B)^{-1} d\mu \\ \\
=  -\frac{\lambda^\alpha}{2\pi i} \int_\Gamma \frac{\mu^\alpha}{1 + \mu} (\lambda \mu - B)^{-1} d\mu, 
\end{array}
$$
implying
$$
\|B^\alpha(\lambda + B)^{-1} \|_{\mL(C([0, T]; X))} \leq C_1 \lambda^{\alpha -1} \int_\Gamma \frac{|\mu|^{\alpha-1}}{1 + |\mu|}  |d\mu| \leq C_2 \lambda^{\alpha -1}. 
$$
We deduce from (\ref{eq13}) that 
$$\|B^\alpha u\|_{C([0, T]; X)} =  C_3 (\lambda^\alpha \|u\|_{C([0, T]; X)} + \lambda^{\alpha-1} \|Bu\|_{C([0, T]; X)}) \quad \forall \lambda \in \R^+. $$
Taking $\lambda =  \|u\|_{C([0, T]; X)}^{-1} \|Bu\|_{C([0, T]; X)}$, we obtain (\ref{eq13A}) and so 
$$
{\stackrel{o}{B}}^{\alpha}_{\infty,1}([0, T]; X)  \subseteq D(B^\alpha).
$$
The inclusion 
$$
D(B^\alpha) \subseteq {\stackrel{o}{C}}^{\alpha}([0, T]; X) 
$$
is much simpler: let $f := B^\alpha u$. Then $f \in C([0, T]; X)$ and, if $t \in [0, T]$, 
$$
u(t) = \frac{1}{\Gamma(\alpha)} \int_0^t (t-\tau)^{\alpha-1} f(\tau) d\tau. 
$$
Then, clearly, $u(0) = 0$. Moreover, if $0 \leq s < t \leq T$, 
$$
\begin{array}{c}
\|u(t) - u(s)\| \leq \frac{1}{\Gamma(\alpha)} (\int_s^t (t-\tau)^{\alpha-1} \|f(\tau)\| d\tau + \int_0^s [(s- \tau)^{\alpha-1} - (t-\tau)^{\alpha-1}] \|f(\tau)\| d\tau) \\ \\
\leq \frac{1}{\Gamma(\alpha+1)} \|f\|_{C([0, T]; X)} [(t-s)^\alpha - (t^\alpha - s^\alpha)] \leq \frac{1}{\Gamma(\alpha+1)} \|f\|_{C([0, T]; X)} (t-s)^\alpha. 
\end{array}
$$

\end{proof}

In order to avoid exceptions we introduce the following spaces:  
\begin{equation}
\begin{array}{c}
B^{1}_{\infty,\infty}([0, T]; X) := \{f \in C([0, T]; X) :  \\ \\
sup_{0 < h < T/2} h^{-1} \|f(. + h) - 2 f + f(. - h) \|_{C([h, T-h]; X)} dh < \infty\}, 
\end{array}
\end{equation}
It is known that, $\forall \epsilon \in (0, 1)$, 
$$
C^1([0, T]; X) \hookrightarrow B^{1}_{\infty,\infty}([0, T]; X) \hookrightarrow C^{1-\epsilon} ([0, T]; X). 
$$
We set
\begin{equation}\label{eq16}
{\stackrel{o}{B}}^{1}_{\infty,\infty}([0, T]; X):= \{f \in C([0, T]; X) : sup_{ t < t+h \leq T} h^{-1} \|\tilde f(t + h) - 2 \tilde f(t) + \tilde f(t- h) \| < \infty \},
\end{equation}
with 
$$
\tilde f(\tau) = \left\{\begin{array}{lll}
f(\tau) & {\rm if }& \tau \in [0, T], \\ \\
0 & {\rm if }& \tau < 0. 
\end{array}
\right. 
$$
${\stackrel{o}{B}}^{1}_{\infty,\infty}([0, T]; X)$ coincides with the set of elements in $B^{1}_{\infty,\infty}([0, T]; X)$ whose trivial extension to $(-\infty, T]$ belongs to $B^{1}_{\infty,\infty}((-\infty, T]; X)$. 

\begin{proposition}\label{pr8A}
${\stackrel{o}{B}}^{1}_{\infty,\infty}([0, T]; X)$ coincides with 
$$\{f \in B^{1}_{\infty,\infty}([0, T]; X) : \sup_{0 < t \leq T} t^{-1} \|f(t)\| < \infty\}.  $$
\end{proposition}
\begin{proof}
Let $f \in {\stackrel{o}{B}}^{1}_{\infty,\infty}([0, T]; X)$ . Let $t \in (0, T]$, $h = 2t$. Then, for some $C \in \R^+$ independent of $t$, 
$$
\|\tilde f(-t + h) - 2 \tilde f(-t) + \tilde f(t - h)\| \leq Ch, 
$$
that is, $\|f(t)\| \leq 2 Ct$. On the other hand, let $f \in B^{1}_{\infty,\infty}([0, T]; X)$ be such that, for some $C \in \R^+$, $\forall t \in [0, T]$, $\|f(t)\| \leq Ct$. Let $h \in \R^+$. Suppose that $t - h < 0 \leq t < t+h \leq T$. Then
$$
\|\tilde f(t+h) - 2 \tilde f(t) + \tilde f(t-h)\| = \|f(t+h) - 2f(t)\| \leq C(2t + h) \leq 3Ch. 
$$
Finally, suppose that $t < 0$, $h \in \R^+$, $0 \leq t + h \leq T$. Then
$$
\|\tilde f(t+h) - 2 \tilde f(t) + \tilde f(t-h)\| = \|f(t+h)\| \leq C(t+h) \leq Ch. 
$$
\end{proof}

More generally, we set, if $m \in \N$, 
\begin{equation}
\begin{array}{c}
B^{m}_{\infty,\infty}([0, T]; X) := \{f \in C^{m-1}([0, T]; X) :  f^{(m-1)} \in B^{1}_{\infty,\infty}([0, T]; X)\},
\end{array}
\end{equation}
\begin{equation}
{\stackrel{o}{B}}^{m}_{\infty,\infty}([0, T]; X):= \{f \in {\stackrel{o}{C}}^{m-1}([0, T]; X) : f^{(m-1)} \in  {\stackrel{o}{B}}^{1}_{\infty,\infty}([0, T]; X)\}. 
\end{equation}

It will be convenient to set, $\forall \gamma \in \R^+$, 
\begin{equation}
\stackrel{o}{{\mathcal C}^\gamma}([0, T]; X) := \left\{\begin{array}{lll}
\stackrel{o}{C^\gamma}([0, T]; X) & {\rm if } & \gamma \not \in \N, \\ \\
\stackrel{o}{B}^\gamma_{\infty, \infty}([0, T]; X) & {\rm if } & \gamma  \in \N. 
\end{array}
\right.
\end{equation}
It is known that, $\forall \alpha, \beta \in \R^+$, with $\alpha < \beta$, $\forall \theta \in (0, 1)$, 
\begin{equation}
(\stackrel{o}{{\mathcal C}^\alpha}([0, T]; X), \stackrel{o}{{\mathcal C}^\beta}([0, T]; X))_{\theta,\infty} = \stackrel{o}{\mathcal C}^{(1-\theta)\alpha + \theta \beta}([0, T]; X)). 
\end{equation}
See for this \cite{Gu1}. As a consequence, we deduce the following

\begin{proposition}\label{pr8}
Let $\alpha, \beta \in \R^+$. Then $B^\alpha$ is a linear and topological isomorphism between $\stackrel{o}{\mathcal C}^{\alpha+\beta}([0, T]; X)$ and $\stackrel{o}{\mathcal C}^{\beta}([0, T]; X)$.
\end{proposition}

\begin{proof}
$B^\alpha$ is a linear and topological isomorphism (isomorphism of Banach spaces) between $D(B^{\alpha + \beta/2})$ and $D(B^{\beta/2})$ and between $D(B^{\alpha + 2\beta})$ and $D(B^{2\beta})$. So, by the interpolation property, it is an isomorphism of Banach spaces between $(D(B^{\alpha + \beta/2}), D(B^{\alpha + 2\beta}))_{1/3, \infty}$ and $(D(B^{\beta/2}), D(B^{2\beta}))_{1/3, \infty}$. By Proposition \ref{pr7} and the reiteration property, 
$$
(D(B^{\beta/2}), D(B^{2\beta}))_{1/3, \infty} = (\stackrel{o}{{\mathcal C}^{\beta/2}}([0, T]; X), \stackrel{o}{{\mathcal C}^{2\beta}}([0, T]; X))_{1/3,\infty} =  {\stackrel{o}{\mathcal C}}^{\beta}([0, T]; X), 
$$
$$
(D(B^{\alpha+\beta/2}), D(B^{\alpha+2\beta}))_{1/3, \infty} = (\stackrel{o}{{\mathcal C}^{\alpha+\beta/2}}([0, T]; X), \stackrel{o}{{\mathcal C}^{\alpha+2\beta}}([0, T]; X))_{1/3,\infty} =  {\stackrel{o}{\mathcal C}}^{\alpha+\beta}([0, T]; X). 
$$
\end{proof}

\begin{proposition}\label{pr10}
Let $\alpha \in \R^+$, $\beta \geq 0$. Then: 

(I) if $\beta \in \N_0$, 
$$
\begin{array}{ccc}
\{u \in D(B^\alpha) : B^\alpha u \in C^\beta([0, T]; X)\} & = &  D(B^{\alpha + \beta}) \oplus \{\sum_{k=0}^{\beta -1} t^{k+\alpha} f_k : f_k \in X\}. 
\end{array}
$$

(II) If $\beta \not \in \N_0$, 
$$
\begin{array}{ccc}
\{u \in D(B^\alpha) : B^\alpha u \in C^\beta([0, T]; X)\} & = &  \stackrel{o}{\mathcal C}^{\alpha+\beta}([0, T]; X) \oplus \{\sum_{k=0}^{[\beta]} t^{k+\alpha} f_k : f_k \in X\}. 
\end{array}
$$

(III) In each case, 
\begin{equation}\label{eq22A}
f_k = \frac{1}{\Gamma (\alpha +k+1)} (B^\alpha u)^{(k)}(0).
\end{equation}
\end{proposition}

\begin{proof}
(I)-(II) We have 
$$
C^\beta([0, T]; X) = D(B^\beta) \oplus \{\sum_{k=0}^{\beta -1} t^k \ g_k: g_k \in X\}, 
$$
in case $\beta \in \N_0$, and 
$$
C^\beta([0, T]; X) = \stackrel{o}{\mathcal C}^\beta([0, T]; X) \oplus \{\sum_{k=0}^{[\beta]} t^k  g_k: g_k \in X\}, 
$$
otherwise. So (I)-(II) follow from Proposition \ref{pr8} and Remark \ref{re9}. 

Moreover, by Remark \ref{re9}, 
\begin{equation}\label{eq23A}
(B^\alpha u)(t) = (B^\alpha v)(t) + \sum_{k=0}^{[\beta]} \frac{\Gamma(\alpha + k+1)}{k!} t^k f_k,
\end{equation}
with $v \in D(B^\beta)$ in case (I), $\stackrel{o}{\mathcal C}^\beta([0, T]; X)$ in case (II).
Finally,  (III) follows from (\ref{eq23A})    and Taylor's formula. 

\end{proof}

\begin{lemma}\label{le12}
Let $\alpha, \beta \in \R^+$,  $u \in D(B^\alpha)$, $B^\alpha u \in C^\beta ([0, T]; X)$, with $\beta \in \R^+$. Let $N$ be the largest integer less than $\beta$. Let $k \in \N_0$, $k \leq N$. Then there exist real numbers $\gamma_{kj}$, independent of $T$ and $u$, and  $v \in D(B^{\alpha + \beta})$ if $\beta \in \N$, $v \in$ $ \stackrel{o}{\mathcal C}^{\alpha+\beta}([0, T]; X)$ if $\beta \not \in \N$, such that
$$
(B^\alpha u)^{(k)}(0) = t^{-k-\alpha} \sum_{j=1}^{N+1} \gamma_{kj} [u(jt) - v(jt)], \quad \forall t \in (0, \frac{T}{N+1}]. 
$$
\end{lemma}

\begin{proof}  We can replace $B^\alpha u)^{(k)}(0)$ with $f_k$ as in the statement of Proposition \ref{pr10}. Then we have
$$
\sum_{k=0}^N (jt)^{k+\alpha} f_k = u(jt) - v(jt), \quad j = 1, \dots, N+1, \quad t \in (0, \frac{T}{N+1}]. 
$$
with $v$ again as in the statement of Proposition \ref{pr10}. The matrix $((jt)^{k+\alpha})_{1 \leq j \leq N+1, 0 \leq k \leq N}$ is invertible and its inverse has the form $(\gamma_{kj} t^{-k-\alpha})_{0 \leq k \leq N, 1 \leq j \leq N+1}$. The conclusion follows.

\end{proof}

\begin{remark}
{In case $N = 0$, we have 
\begin{equation}\label{eq24A}
(B^\alpha u)(0) = t^{-\alpha} \Gamma (\alpha +1) [u(t) - v(t)], 
\end{equation}
in case $N = 1$, we have
\begin{equation}\label{eq25}
\left\{\begin{array}{l}
(B^\alpha u)(0) = t^{-\alpha} \frac{\Gamma(\alpha +1)}{2} \{2^{\alpha +1} [u(t) - v(t)] - [u(2t) - v(2t)]\}, \\ \\
(B^\alpha u)'(0) = t^{-\alpha-1} \frac{\Gamma(\alpha +2)}{2}\{- 2^\alpha [u(t) - v(t)] + [u(2t) - v(2t)]\}. 
\end{array}
\right. 
\end{equation}

}
\end{remark}

Now we want to study the following abstract equation:

\begin{equation}\label{eq20}
B^\alpha u(t) - A u(t) = f(t), \quad t \in [0, T]. 
\end{equation}
We assume that $f \in C([0, T]; X)$ and that $A : D(A) (\subset X) \to X$ is a linear (usually unbounded) operator, $\alpha \in \R^+$. 

\begin{definition}
A strict solution of (\ref{eq20}) is a function $u$ belonging to $C([0, T]; X)$, such that $u \in D(B^\alpha)$, $u(t) \in D(A)$ for every $t \in [0, T]$ and (\ref{eq20}) holds for every $t$ in $[0, T]$. 

\end{definition}

The main tool for our study will be the Da Prato-Grisvard's theory developed in \cite{DaGr1}. We are going to describe one of their results, following the presentation in \cite{Gu2}. 

The introduce the following setting:

\medskip

{\it (A1)  $\mX$ is a complex Banach space, $\mD$ and $\mA$ are linear (usually unbounded) operators in $\mX$, with domains $D(\mD)$, $D(\mA)$ respectively; 

(A2) for some $\phi_0, \phi_1 \in (0, \pi)$, $\mD$ and $-\mA$ are operators of type $\phi_0$, $\phi_1$; 

(A3) $\phi_0 + \phi_1 < \pi$; 

(A4) $0 \in \rho(\mD)$; 

(A5) $\mD$ and $\mA$ have commuting resolvents in the following sense: if $\lambda \in \rho(\mD)$ and $\mu \in \rho(\mD)$, 
$$(\lambda - \mD)^{-1} (\mu - \mA)^{-1} = (\mu - \mA)^{-1} (\lambda - \mD)^{-1}.$$
}

\medskip

We want to study the equation

\begin{equation}\label{eq21}
(\mD - \mA)u = f, 
\end{equation}
with $f \in \mX$. 

A strict solution of (\ref{eq21}) is an element $u$ belonging to $D(\mD) \cap D(\mA)$ such that (\ref{eq21}) holds. 

\medskip

\begin{remark}\label{re13}
{\rm The condition "$-\mA$ is an operator of type $\phi_1$ for some $M > 0$" is equivalent to 
$"\{\lambda \in \C \setminus \{0\} : |Arg(\lambda)| < \pi - \phi_1\} \subseteq \rho (\mA)$ and  $\forall \epsilon \in (0, \pi - \phi_1)$ there exists $M(\epsilon) \in \R^+$ such that, if $\lambda \in \C \setminus \{0\}$ and $|Arg(\lambda)| \leq \pi - \phi_1 - \epsilon$, 
$\|\lambda (\lambda - A)^{-1}\|_{\mL(X)} \leq M(\epsilon)$". 

}
\end{remark}

We fix a piecewise regular, simple path $\gamma$, describing $\{\lambda \in \C \setminus \{0\}: |Arg(\lambda)| = \phi_2, |\lambda| \geq r\} \cup \{\lambda \in \C \setminus \{0\}: |Arg(\lambda)| \leq \phi_2, |\lambda| = r\}$, oriented from $\infty e^{-i\phi_2}$ to 
$\infty e^{i\phi_2}$, with $\phi_0 < \phi_2 < \pi - \phi_1$, $r \in \R^+$, such that $\{\lambda \in \C : |\lambda| \leq r\} \subseteq \rho(\mD)$. It is easily seen that $\gamma$ is contained in $\rho(\mD) \cap \rho(\mA)$. So the following element $S$ of $\mL(\mX)$
is well defined: 
\begin{equation}\label{eq22}
S: = \frac{1}{2\pi i} \int_\gamma (\mD - \lambda)^{-1} (\lambda - \mA)^{-1} d\lambda. 
\end{equation}

The following result holds:

\begin{theorem}\label{th12}
Suppose that the assumptions (A1)-(A5) are fulfilled. Then:

(I) $\forall f \in \mX$ (\ref{eq21}) has at most one strict solution $u$.

(II) If such solution $u$ exists, then $u = Sf$. 

(III) In case, for some $\theta \in (0, 1)$,  $f$ belongs to the real interpolation space $(\mX, D(\mD))_{\theta,\infty}$($(\mX, D(\A))_{\theta,\infty}$) the strict solution $u$ exists. In this situation even  $\mD u$ and $\mA u$ belong to  $(\mX, D(\mD))_{\theta,\infty}$($(\mX, D(\A))_{\theta,\infty}$).

\end{theorem}

For future use, we state the following

\begin{proposition} \label{pr13}

Let $A$  a linear operator of type $\phi$, for some $\phi \in (0, \pi)$, in the Banach space $X$. Let $\theta \in (0, 1)$. Then:

(I) the interpolation space $(X, D(A))_{\theta,\infty}$ coincides with $\{f \in X : sup_{t > 0} t^\theta \|A(t+A)^{-1} f\| < \infty\}$. Moreover, an equivalent norm in $(X, D(A))_{\theta,\infty}$ is 
$$
f \to \max\{\|f\|,  \sup_{t > 0} t^\theta \|A(t+A)^{-1} f\|\}. 
$$

(II) Let $\mA$ be the operator in $C([0, T]; X)$ defined as follows: 
\begin{equation}\label{eq23}
\left\{ \begin{array}{l}
D(\mA) = C([0, T]; D(A)), \\ \\
(\mA u)(t) = A u(t). 
\end{array}
\right. 
\end{equation}
Then $\rho(A) \subseteq \rho(\mA)$ 
$$
[(\lambda - \mA)^{-1}f] (t) = (\lambda - A)^{-1} f(t), 
$$
for every $\lambda \in \rho(A)$, $f \in C([0, T]; X)$, $t \in [0, T]$. $\mA$  a linear operator of type $\phi$ and the interpolation space $(C([0, T]; X); C([0, T]; D(A)))_{\theta,\infty}$ coincides with
$$C([0, T]; X) \cap B([0, T]; (X, D(A))_{\theta,\infty}). $$
\end{proposition}

\begin{proof}
For (I), see \cite{Tr1}, Theorem 1.14.2. 

(II) follows quite simply from (I). We prove only the identity between $(C([0, T]; X); C([0, T]; D(A)))_{\theta,\infty}$ and $C([0, T]; X) \cap B([0, T]; (X, D(A))_{\theta,\infty})$. If $f \in (C([0, T]; X); C([0, T]; D(A)))_{\theta,\infty}$, by (I), 
$$
sup_{t > 0} t^\theta \|\mA(t+ \mA)^{-1} f\|_{C([0, T]; X)} < \infty
$$
and
$$
\begin{array}{c}
sup_{t > 0} t^\theta \|\mA(t+ \mA)^{-1} f\|_{C([0, T]; X)} = sup_{t > 0} \sup_{s \in [0, T]} t^\theta \|A(t+ A)^{-1} f(s)\| \\ \\
= \sup_{s \in [0, T]} sup_{t > 0} t^\theta \|A(t+ A)^{-1} f(s)\|. 
\end{array}
$$
So 
$$
\begin{array}{c}
(C([0, T]; X); C([0, T]; D(A)))_{\theta,\infty} = (C([0, T]; X); D(\mA))_{\theta,\infty} \\ \\
 = \{f \in C([0, T]; X) : \sup_{s \in [0, T]} \|f(s)\|_{(X, D(A))_{\theta,\infty}} < \infty\}. 
 \end{array}
$$

\end{proof}

A simple application of Theorem \ref{th12} is the following: 

\begin{proposition}\label{pr15}
Let $\alpha \in (0, 2)$. Consider equation (\ref{eq20}),supposing that $-A$ is of type $\phi$, for some $M \in \R^+$, $\phi \in (0, (1 - \frac{\alpha}{2})\pi)$; then:

(I) (\ref{eq20}) has, at most, one strict solution, for every $f \in C([0, T]; X)$. 

(II) If $f \in$  $\stackrel{o}{{\mathcal C}^\beta}([0, T]; X)$, for some $\beta \in (0, \alpha)$, the strict solution $u$ exists and is such that $B^\alpha u$, $Au$ belong to $\stackrel{o}{{\mathcal C}^\beta}([0, T]; X)$. 

(III) If $f \in  C([0, T]; X) \cap B([0, T]; (X, D(A))_{\theta,\infty})$ (that is, $f$ is bounded with values in $(X, D(A))_{\theta,\infty}$) for some $\theta \in (0, 1)$, the strict solution $u$ exists and is such that $B^\alpha u$, $Au$ are bounded with values in $(X, D(A))_{\theta,\infty}$. 
\end{proposition}

\begin{proof} We set $\mX:= C([0, T]; X)$, $\mD:= B^\alpha$. $\mA$ is defined as in (\ref{eq23}). 
Then $\rho(A) \subseteq \rho(\mA)$ and  $- \mA$ is of type $\phi$ and  $\mD$ and $\mA$ have commuting resolvents: if $\lambda \in \C \setminus \{0\}$, $|Arg(\lambda)| > \frac{\alpha \pi}{2}$, $\mu \in \rho(A)$ and $f \in C([0, T]; X)$, we have, for every $t \in [0, T]$,
$$
\begin{array}{c}
[(\lambda - B^\alpha)^{-1} (\mu - \mA)^{-1}f](t) = \frac{1}{2\pi i} \int_0^t (\int_{\Gamma (\lambda)} \frac{e^{\nu(t-s)}}{\lambda - \nu^\alpha} d\nu) (\mu - A)^{-1} f(s) ds \\ \\
= (\mu - A)^{-1} [\frac{1}{2\pi i} \int_0^t (\int_{\Gamma (\lambda)} \frac{e^{\nu(t-s)}}{\lambda - \nu^\alpha} d\nu)  f(s) ds] \\ \\
= [(\mu - \mA)^{-1} (\lambda - B^\alpha)^{-1}f](t). 
\end{array}
$$
Now, it is known that, in order that $(A5)$ holds, it suffices that there exist $\lambda_0 \in \rho(\mD)$ and $\mu_0 \in \rho(\mA)$ such that
$$
(\lambda_0 - \mD)^{-1} (\mu_0 - \mA)^{-1} = (\mu_0 - \mA)^{-1} (\lambda_0 - \mD)^{-1}.
$$
See, for this, .... . So Theorem \ref{th12} is applicable. 

In order to show (II), we observe that, by Proposition \ref{pr7} and \cite{Gu1}, we have
$$
\begin{array}{c}
\stackrel{o}{{\mathcal C} ^\beta}([0, T]; X) = (C([0, T]; X), \stackrel{o}{B}^{\alpha}_{\infty,1}([0, T]; X))_{\beta/\alpha , \infty} \\ \\
  \subseteq (C([0, T]; X), D(B^\alpha))_{\beta/\alpha , \infty} \\ \\
  \subseteq 
(C([0, T]; X), \stackrel{o}{C}^{\alpha}([0, T]; X))_{\beta/\alpha , \infty} =  \hspace{.1in} \stackrel{o}{{\mathcal C}^\beta}([0, T]; X). 
\end{array}
$$
Finally, (III) can be obtained applying Proposition \ref{pr13} (II). 

\end{proof}

Now we try to write explicitly the operator $S$ in (\ref{eq22}) in our particular situation: we fix a piecewise regular, simple path $\gamma$ describing $\{\lambda \in \C : |\lambda| \geq 1, |Arg(\lambda)| = \phi_2\} \cup \{\lambda \in \C : |\lambda| = 1, |Arg(\lambda)| \leq \phi_2\}$, oriented 
from $\infty e^{-i\phi_2}$ to $\infty e^{i\phi_2}$, with $\frac{\alpha \pi}{2} < \phi_2 < \pi - \phi$. Then we have, $\forall f \in C([0, T]; X)$, $t \in [0, T]$, 
$$
Sf(t) = \frac{1}{(2\pi i)^2} \int_\gamma (\int_0^t (\int_{\Gamma(\lambda)} \frac{e^{\mu(t-s)}}{\mu^\alpha - \lambda } d\mu) (\lambda - A)^{-1} f(s) ds ) d\lambda. 
$$
We fix $\theta > \frac{\pi}{2}$, such that $\alpha \theta < \phi_2$, and take a piecewise regular, simple path $\Gamma$ describing $\{\mu \in \C : |\mu| \geq 2, |Arg(\mu)| = \theta\} \cup \{\mu \in \C : |\lambda| = 2, |Arg(\mu)| \leq \theta\}$, oriented 
from $\infty e^{-i\theta}$ to $\infty e^{i\theta}$. Then, by Proposition \ref{pr5}, we can take $\Gamma(\lambda) = \Gamma$ for every $\lambda \in \gamma$. 

We have 
$$
\begin{array}{c}
 \int_\gamma (\int_0^t (\int_{\Gamma} \frac{e^{Re(\mu)(t-s)}}{|\mu^\alpha - \lambda |} |d\mu|)| \|(\lambda - A)^{-1} f(s\|) ds ) |d\lambda| \\ \\
 \leq C_1 \int_{\gamma \times \Gamma} \frac{|e^{Re(\mu)t} - 1|}{|Re(\mu)| |\lambda| |\mu^\alpha - \lambda|} |d\lambda| |d\mu| \\ \\
 \leq C_2(  \int_{\gamma \times \Gamma} \min\{t, |Re(\mu)|^{-1}\} |\lambda|^{-1} (|\mu|^{\alpha} + |\lambda|)^{-1} |d\lambda| |d\mu| \\ \\
 \leq  C_3 \int_{[1 , \infty)^2} \frac{1}{r \rho(r + \rho^\alpha)} dr d\rho \leq C_4. 
 \end{array}
$$
So, by the theorems of Fubini and Cauchy, 
\begin{equation}\label{eq24}
\begin{array}{c}
Sf(t) = \frac{1}{(2\pi i)^2} \int_0^t (\int_{\Gamma} (\int_\gamma (\mu^\alpha - \lambda)^{-1} (\lambda - A)^{-1} d\lambda ) e^{\mu(t-s)} d\mu) f(s) ds \\ \\
= \frac{1}{2\pi i} \int_0^t  (\int_{\Gamma} e^{\mu(t-s)} (\mu^\alpha - A)^{-1} d\mu)  f(s) ds. 
\end{array}
\end{equation}

Now we look for conditions assuring that (\ref{eq21}) has a unique strict solution $u$ such that $B^\alpha u$ and $Au$ are in $C^\beta ([0, T]; X)$, for a fixed $\beta \in (0, \alpha) \setminus \{1\}$.

\begin{lemma}\label{le18}
Let  $0 < \beta <  \alpha$, $\beta \not \in \N$, and let $A$ be a closed linear operator in $X$. Let $u \in D(B^\alpha) \cap C^\beta([0, T]; D(A))$ be such that $B^\alpha u \in C^\beta([0, T]; X)$. Let $k\in \N_0$, $j < \beta$. Then,
$D^k(B^\alpha u)(0)$ belongs to the interpolation space 
$(X, D(A))_{\frac{\beta -k}{\alpha}, \infty}$. 
\end{lemma}

\begin{proof} In order to show the assertion, we shall employ the definition of the space $(X, D(A))_{\theta, \infty}$ by the K-Method (see \cite{Tr1}, Chapter 1.3): for every $s \in \R^+$, we define the functional $K(s,.)$ in $X$ as follows: 
$$
K(s,x) := \inf \{ \|x - y\| + s \|y\|_{D(A)} : y \in D(A)\}.
$$
Then $x \in (X, D(A))_{\theta, \infty}$ if and only if, for some $C \in \R^+$, $\forall s \in \R^+$, $K(s,x) \leq C s^\theta$. By Lemma \ref{le12}, if $k \in \N_0$ and $k \leq [\beta]$, 
$$
(B^\alpha u)^{(k)}(0) = t^{-k-\alpha} \sum_{j=1}^{N+1} \gamma_{kj} [u(jt) - v(jt)], \quad \forall t \in (0, \frac{T}{N+1}], 
$$
for certain constants $\gamma_{kj}$ and some $v$ in $\stackrel{o}{\mathcal C}^{\alpha+\beta}([0, T]; X)$. If $t \in (0, \frac{T}{N+1}]$, we set
$$
w_k(t) := t^{-k-\alpha} \sum_{j=1}^{N+1} \gamma_{kj} u(jt). 
$$
Then
$$
\|(B^\alpha u)^{(k)}(0) - w_k(t)\| = t^{-k-\alpha} \|\sum_{j=1}^{N+1} \gamma_{kj} v(jt)\| \leq C_1 t^{\beta - k}, 
$$
as $v \in$  $\stackrel{o}{\mathcal C}^{\alpha+\beta}([0, T]; X)$. This is true even in case $\alpha + \beta \in \N$, by Proposition \ref{pr8A}. 
Moreover, as $u \in C^\beta([0, T]; D(A))$ and $Au^{(j)}(0) = 0$ if $j < \beta$, we have
$$
\|w_k(t)\|_{D(A)} \leq C_2 t^{\beta - \alpha-k}, \quad \forall t \in [0, \frac{T}{N+1}]. 
$$
Let $s \in \R^+$. We set
$$
z_k(s) := \left\{\begin{array}{lll}
w_k(s^{1/\alpha}) & {\rm if } & 0 < s \leq  (\frac{T}{N+1})^\alpha, \\ \\
0  & {\rm if } & s > (\frac{T}{N+1})^\alpha . 
\end{array}
\right. 
$$
We deduce, $\forall s \in \R^+$, 
$$
K(s, B^\alpha u(0)) \leq \|B^\alpha u(0) - z_k(s)\| + s \|z_k(s)\|_{D(A)} \leq C_3s^{\frac{\beta - k}{\alpha}}, 
$$
for some $C_3 \in \R^+$ independent of $s$. 
\end{proof}

\begin{proposition} \label{pr19}
Let $\alpha \in (0, 2)$. Consider equation (\ref{eq20}),supposing that $-A$ is of type $\phi$, for some $\phi \in (0, (1 - \frac{\alpha}{2})\pi)$. Let $\beta \in (0, \alpha)$, $\beta \neq 1$, $f \in C([0, T]; X)$. Then the following conditions are necessary and sufficient, in order that there exist a (unique) strict solution $u$ such that $B^\alpha u$, $\mA u$ belong to $C^\beta([0, T]; X)$: 

(a) $f \in C^\beta([0, T]; X)$; 

(b) if $k \in \N_0$ and $k < \beta$, $f^{(k)}(0)$ belongs to the interpolation space $(X, D(A))_{\frac{\beta - k}{\alpha}, \infty}$. 
\end{proposition}

\begin{proof} The necessity of (a) is obvious. 

Moreover, by Lemma \ref{le18}, if $k \in \N_0$ and $k < \beta$, $(B^\alpha u)^{(k)}(0) \in (X, D(A))_{\frac{\beta - k}{\alpha}, \infty}$. As $Au^{(k)}(0) = 0$ if $k < \beta$, we deduce
$$
(B^\alpha u)^{(k)}(0) = f^{(k)}(0).
$$
We conclude that even (b) is necessary. 

We show that (a)-(b) are sufficient. If a strict solution $u$ exists, then 
$$u = Sf = S(f- \sum_{k < \beta} t^k f^{(k)}(0)) + \sum_{k < \beta}S(t^k f^{(k)}(0) )).$$
 Now, $f - \sum_{k < \beta} t^k f^{(k)}(0)$ belongs to the space $\stackrel{o}{C^\beta}([0, T]; X)$. So, by Proposition \ref{pr15}, $S(f- \sum_{k < \beta} t^k f^{(k)}(0)) $ is a strict solution. Moreover, $B^\alpha S(f- \sum_{k < \beta} t^k f^{(k)}(0))$ and $A S(f- \sum_{k < \beta} t^k f^{(k)}(0))$ both belong to $\stackrel{o}{C^\beta}([0, T]; X)$. 
 
 Now we consider $S(f(0))$. By Proposition \ref{pr15} (III), this also is a strict solution. So, from $B^\alpha S(f(0)) = 
A S(f(0))  + f(0)$, in order to reach the conclusion, we have only to show that $A S(f(0))$ belongs to $C^\beta([0, T]; X)$. By (\ref{eq24}), we have
$$S(f(0))(t) = \frac{1}{2\pi i} \int_0^t  (\int_{\Gamma} e^{\mu s} (\mu^\alpha - A)^{-1}f(0) d\mu)   ds. $$
We set, for $t \in (0, T]$, 
\begin{equation}\label{eq31B}
g(t) = \frac{1}{2\pi i} \int_{\Gamma} e^{\mu t} (\mu^\alpha - A)^{-1}f(0) d\mu . 
\end{equation}
We deformate the integration path $\Gamma$, replacing it with $\Gamma_0$, with $\Gamma_0$ describing $\mu \in \C \setminus \{0\} : |Arg(\mu)| = \theta\}$, oriented from $\infty e^{-i\theta}$ to $\infty e^{-i\theta}$, for some $\theta \in (\frac{\pi}{2}, \frac{\pi - \phi_1}{\alpha})$. We have, employing Proposition \ref{pr13}, 
\begin{equation}\label{eq31A}
\begin{array}{c}
\|Ag(t)\| = \|\frac{1}{2\pi i} \int_{\Gamma} e^{\mu t} A(\mu^\alpha - A)^{-1}f(0) d\mu\| \leq C_1 \int_{\Gamma} e^{Re(\mu) t} |\mu|^{-\beta}  |d\mu| \\ \\
\leq C_2 \int_{\R^+} e^{cos(\theta) \rho t} \rho^{-\beta} d\rho = C_3 t^{\beta -1}. 
\end{array}
\end{equation}
We suppose first that $\beta < 1$. If $0 < s < t \leq T$, 
$$
\|A S(f(0))(t) - A S(f(0))(s)\| \leq \int_s^t \|Ag(\tau)\| d\tau \leq C_4 (t^\beta - s^\beta) \leq C_5 (t-s)^\beta. 
$$
So $S(f(0)) \in C^\beta([0, T]; D(A))$. Suppose now that $1 < \beta < \alpha < 2$.
 
If $\beta > 1$, (\ref{eq31A}) implies that $Ag \in C([0, T]; X)$, so that $S(f(0)) \in C^1([0, T]; D(A))$. If $t \in (0, T]$, we have
$$
A g'(t) = \frac{1}{2\pi i} \int_{\Gamma} e^{\mu t} \mu A(\mu^\alpha - A)^{-1}f(0) d\mu 
$$
and, arguing as before, we get 
$$
\|A g'(t)\| \leq C_5 t^{\beta -2},
$$
or
$$
\|(Au)''(t)\| \leq C_5 t^{\beta -2}. 
$$
This implies that $(Au)' \in C^{\beta-1}([0, T]; X)$ and $Au \in C^{\beta}([0, T]; X)$. So in  any case $S(f(0)) \in C^\beta([0, T]; D(A))$. 

We examine $S(t f'(0))$, of course only in case $1 < \beta < \alpha < 2$, assuming $f'(0) \in (X, D(A))_{\frac{\beta - 1}{\alpha}, \infty}$. This also is a strict solution, again by Proposition \ref{pr15} (III), and here also it suffices to
 show that $A S(tf'(0))$ belongs to $C^\beta([0, T]; X)$. By (\ref{eq24}), we have
 $$
 S(t f'(0)(t) = \frac{1}{2\pi i} \int_0^t s (\int_\Gamma e^{\mu(t-s)} (\mu^\alpha - A)^{-1}f'(0) d\mu) ds =  \int_0^t (t-s) (\frac{1}{2\pi i} \int_\Gamma e^{\mu s} (\mu^\alpha - A)^{-1}f'(0) d\mu) ds. 
 $$
 We define $g(t)$ as in (\ref{eq31B}).  Then, arguing as before, we get
 $$
 \|Ag(t)\| \leq C_1 t^{\beta - 2}. 
 $$
 or 
 $$
 \|A S(t f'(0))''(t)\| \leq C_1 t^{\beta - 2}. 
 $$
 We deduce that $A S(t f'(0))' \in C^{\beta - 1}([0, T]; X)$, so that $A S(t f'(0)) \in C^{\beta}([0, T]; X)$. 
\end{proof}
\begin{proposition}\label{pr20}
The conclusions of Propositions \ref{pr15} and \ref{pr19} are still valid if we assume the following:

\medskip
(H) "there exist $\lambda_0 \geq 0$ and $\phi \in (\frac{\alpha \pi}{2}, \pi)$ such that $\{\lambda \in \C \setminus \{\lambda_0\} : |Arg(\lambda - \lambda_0)| < \phi\} \subseteq \rho(A)$.  Moreover, $\forall \epsilon \in (0, \phi)$ there exists $M(\epsilon) \in \R^+$, such that, if $|Arg(\lambda - \lambda_0)| \leq \phi - \epsilon$, $\|(\lambda - \lambda_0)(\lambda - A)^{-1}\|_{\mL(X)} \leq M(\epsilon)$". 

\end{proposition}

\begin{proof}
We write (\ref{eq20}) in the equivalent form

\begin{equation}\label{eq26}
(B^\alpha - \lambda_0)u(t) - (A - \lambda_0)u(t) = f(t), \quad t \in [0, T]. 
\end{equation}
By Proposition \ref{pr5}, $B^\alpha - \lambda_0$ is of type $\frac{\alpha \pi}{2}$  and $-(A - \lambda_0) = \lambda_0 - A$ is of type $\pi - \phi$, for some $M$ positive. The interpolation spaces involved do not change. 
\end{proof}

\begin{remark}\label{re21}
{\rm Suppose that the assumptions of Proposition \ref{pr20} are satisfied. Then the solution $u$ to (\ref{eq26}) can be represented in the form
$$
u = \frac{1}{2\pi i} \int_\gamma (B^\alpha - \lambda_0 - \lambda)^{-1} (\lambda + \lambda_0 - \mA)^{-1} d\lambda = \frac{1}{2\pi i} \int_{\gamma + \lambda_0} (B^\alpha - \lambda)^{-1} (\lambda  - \mA)^{-1} d\lambda
$$
with $\mA$ as in (\ref{eq23}) and $\gamma$ piecewise regular, simple, describing $\{\lambda \in \C \setminus \{0\}: |Arg(\lambda)| = \phi_2, |\lambda| \geq r\} \cup \{\lambda \in \C \setminus \{0\}: |Arg(\lambda)| \leq \phi_2, |\lambda| = r\}$, oriented from $\infty e^{-i\phi_2}$ to 
$\infty e^{i\phi_2}$, with $\frac{\alpha \pi}{2} < \phi_2 <  \phi$, $r \in \R^+$. Next, we can take $(\lambda - B)^{-1}$ in the form (\ref{eq9}), with 
$\Gamma (\lambda) = \Gamma$,  $\Gamma$ piecewise regular, simple, describing $\{\mu \in \C \setminus \{0\}: |Arg(\mu)| = \theta, |\mu| \geq \rho\} \cup \{\mu \in \C \setminus \{0\}: |Arg(\mu)| \leq \theta, |\mu| = \rho\}$, with $\frac{\pi}{2} < \theta < \frac{\phi_2}{\alpha}$, $\rho^\alpha > r$. So we have, for $t \in [0, T]$, 
$$
\begin{array}{c}
u(t) = \frac{1}{(2\pi i)^2} \int_0^t (\int_{\Gamma} (\int_{\gamma + \lambda_0} (\mu^\alpha - \lambda)^{-1} (\lambda - A)^{-1} d\lambda ) e^{\mu(t-s)} d\mu) f(s) ds \\ \\
= \frac{1}{2\pi i} \int_0^t  (\int_{\Gamma} e^{\mu(t-s)} (\mu^\alpha - A)^{-1} d\mu)  f(s) ds. 
\end{array}
$$

}
\end{remark}

Now we consider the case $\alpha \in (1, 2)$. In this case, the Riemann-Liouville derivative of order $\alpha$  does not coincide with $B^\alpha$. In fact by (\ref{eq6}), in this case, the Riemann-Liouville derivative $g$ of $f$ should be
$$
g = D_t^2 (\frac{1}{\Gamma (\alpha -1)} \int_0^t (t-s)^{1-\alpha} f(s) ds) = D_t^2(B^{\alpha -2} f)(t), 
$$
which should hold for $f \in C([0, T]; X)$ such that $B^{\alpha -2} f \in C^2([0, T]; X)$. Now, $B^{\alpha -2} f(0) = 0$, so that this implies that $B^{\alpha -2} f \in D(B)$ and we get
\begin{equation}
g = D_t^2(B^{\alpha -2} f) = D_t(B^{\alpha -1} f). 
\end{equation}
So, if $\alpha \in (1, 2)$, we can define the Riemann-Liouville derivative $\mL ^\alpha f$ of $f$ as follows:
\begin{equation}\label{eq39}
\left\{\begin{array}{l}
D(\mL ^\alpha ):= \{f \in C([0, T]; X) : B^{\alpha -1} f \in C^1([0, T]; X)\}, \\ \\
\mL ^\alpha f = D_t (B^{\alpha -1} f). 
\end{array}
\right. 
\end{equation}
As 
$$C^1([0, T]; X) = D(B) \oplus \{t \to f_0 : f_0 \in X\}$$
from Remark \ref{re9} we deduce 
\begin{equation}\label{eq40}
D(\mL ^\alpha ) = D(B^\alpha) \oplus \{t \to t^{\alpha-1}  f_0 : f_0 \in X\}. 
\end{equation}
 
 Now we study the system
 
 \begin{equation}\label{eq37}
 \left\{\begin{array}{ll}
\mL ^\alpha u(t) - Au(t) = f(t), & t \in [0, T], \\ \\
B^{\alpha - 1}u(0) = g_0,
\end{array}
\right.
 \end{equation}
with $\alpha \in (1, 2)$. The following result holds: 

\begin{theorem}\label{th23} We consider system (\ref{eq37}) with the following assumptions: 

(a) $\alpha \in (1, 2)$; 

(b) there exist $\lambda_0 \geq 0$ and $\phi \in (\frac{\alpha \pi}{2}, \pi)$ such that $\{\lambda \in \C \setminus \{\lambda_0\} : |Arg(\lambda - \lambda_0)| < \phi\} \subseteq \rho(A)$.  Moreover, $\forall \epsilon \in (0, \phi)$ there exists $M(\epsilon) \in \R^+$, such that, if $|Arg(\lambda - \lambda_0)| \leq \phi - \epsilon$, $\|(\lambda - \lambda_0)(\lambda - A)^{-1}\|_{\mL(X)} \leq M(\epsilon)$.

Let $\beta \in (0, \alpha) \setminus \{1, \alpha -1\}$. Then the following conditions are necessary and sufficient, in order that (\ref{eq37}) have a unique solution $u$, such that $\mL ^\alpha u$, $Au$ belong to $C^\beta([0, T]; X)$: 

(I) if $k \in \N_0$ and $k < \beta$, $f^{(k)}(0) \in (X, D(A))_{\frac{\beta - k}{\alpha}, \infty}$; 

(II) if $\beta < \alpha - 1$, $g_0 \in (X, D(A))_{\frac{\beta +1}{\alpha}, \infty}$; 

(III) if $\beta > \alpha - 1$, $g_0 = 0$. 

\end{theorem}

\begin{proof} By Proposition (\ref{pr10}), 
$$
\begin{array}{c}
\{u \in D(\mL ^\alpha) : \mL ^\alpha u \in C^\beta([0, T]; X) \} = \{u \in D(B^{\alpha - 1}) : B^{\alpha - 1}u \in C^{1+\beta}([0, T]; X)\} \\ \\
=   \hspace{.1in} \stackrel{o}{\mathcal C}^{\alpha+\beta}([0, T]; X) \oplus \{\sum_{k=0}^{[\beta]+1} t^{k+\alpha-1} f_k : f_k \in X\}
\end{array}
$$
So, if $u$ is a solution with the required regularity,  we have 
\begin{equation}\label{eq38}
u(t) = v(t) + \sum_{k < \beta+1} t^{k+\alpha-1} f_k
\end{equation}
with $v \in$  $\stackrel{o}{\mathcal C}^{\alpha+\beta}([0, T]; X)$  and
$$
B^{\alpha -1}u(t) = B^{\alpha -1}v(t) + \sum_{k < \beta+1} \frac{\Gamma (\alpha + k)}{k!} s^{k} f_k,
$$
implying
$$
g_0 = B^{\alpha -1}u(0) = \Gamma(\alpha) f_0. 
$$
Observe that, if $f_0 \neq 0$, $u$ cannot belong to $C^\gamma ([0, T]; X)$ for every $\gamma > \alpha - 1$. So, if $\beta > \alpha - 1$, if we want $u \in C^\beta ([0, T]; D(A))$ and so $u \in C^\beta ([0, T]; X)$, we need to impose $g_0 = 0$. 
In this case, $u \in D(B^\alpha)$,  $\mL ^\alpha u = B^\alpha u$ and we are reduced to (\ref{eq20}). So Proposition \ref{pr19} is applicable.

We have reached the following conclusion: 

\medskip

{ \it ($\gamma_1$) The claim of Theorem \ref{th23} holds true if $g_0 = 0$ and, in general, if $\beta > \alpha - 1$.}

\medskip

From now on, we assume $\beta < \alpha - 1$. Observe that this implies $\beta < 1$, so that (\ref{eq38}) reduces to
$$
u(t) = v(t) +  t^{\alpha-1} f_0 + t^\alpha f_1. 
$$
Now we show that

\medskip

{ \it ($\gamma_2$) If $\beta < \alpha - 1$ and a solution with the declared regularity exists, $g_0 \in (X, D(A))_{\frac{\beta +1}{\alpha}, \infty}$, $f(0) \in  (X, D(A))_{\frac{\beta}{\alpha}, \infty}$.  }

\medskip

We have already observed that $g_0 = \Gamma(\alpha) f_0$. We have also
$$
(B^{\alpha -1}u)'(t) = B^\alpha v(t) + \Gamma(\alpha + 1) f_1, 
$$
implying
$$
f(0) = (B^{\alpha -1}u)'(0) - Au(0) = \Gamma(\alpha + 1) f_1. 
$$
Therefore, in order to show $(\gamma_2)$, we can show that, for $k = 0, 1$, $f_k \in (X, D(A))_{\frac{\beta +1-k}{\alpha}, \infty}$. This can be done employing arguments similar to those in the proof of Lemma \ref{le18}: we have, for $t \in (0, \frac{T}{2}]$, 
$$
\begin{array}{c}
f_0 = t^{1-\alpha} [2u(t) - 2^{1-\alpha} u(2t)] - t^{1-\alpha} [2v(t) - 2^{1-\alpha} v(2t)], \\ \\
f_1 = t^{-\alpha} [2^{1-\alpha} u(2t) - u(t)] - t^{-\alpha} [2^{1-\alpha} v(2t) - v(t)]. 
\end{array}
$$
So we set
$$
\begin{array}{c}
w_0(t)  = t^{1-\alpha} [2u(t) - 2^{1-\alpha} u(2t)], \\ \\
w_1(t) = t^{-\alpha} [2^{1-\alpha} u(2t) - u(t)], 
\end{array}
$$
deducing, as $v \in$  $\stackrel{o}{\mathcal C}^{\alpha+\beta}([0, T]; X)$, for $k \in \{0, 1\}$, 
$$
\|f_k - w_k(t)\| \leq C_k t^{\beta + 1 -k}, \|w_k(t)\|_{D(A)} \leq C_k t^{\beta - \alpha + 1 - k}. 
$$
So, setting 
$$
z_k(s) := \left\{\begin{array}{lll}
w_k(s^{1/\alpha}) & {\rm if } & 0 < s \leq  (\frac{T}{2})^\alpha, \\ \\
0  & {\rm if } & s > (\frac{T}{2})^\alpha , 
\end{array}
\right. 
$$
we deduce
$$
K(s, f_k) \leq 2 C_k s^{\frac{\beta+1-k}{\alpha}}.
$$

\medskip

{ \it ($\gamma_3$) Let $\alpha \in (1, 2)$, $\beta \in (0, \alpha - 1)$. Consider  system (\ref{eq37}), with $f \equiv 0$, $g_0 \in (X, D(A))_{\frac{\beta +1}{\alpha}, \infty}$. Then there exists a  solution $u$ such that $\mL^\alpha u$, $Au$ belong to 
$C^\beta([0, T]; X)$.  }

\medskip

We can try to draw a formula for a possible solution employing the Laplace transform  and Proposition \ref{pr7A}: it should be (at least formally, indicating with $\tilde u$ la Laplace transform of $u$)
$$
0 = \lambda \tilde {B^{\alpha-1}u}(\lambda) - B^{\alpha-1}u(0) - A \tilde u(\lambda) = \lambda^\alpha \tilde u (\lambda) - g_0 - A \tilde u(\lambda), 
$$
hence
$$
\tilde u(\lambda) = (\lambda^\alpha - A)^{-1} g_0. 
$$
The inverse formula of the Laplace transform takes to 
$$
u(t) = \frac{1}{2\pi i} \int_\Gamma e^{\lambda t} (\lambda^\alpha - A)^{-1} g_0 d\lambda, 
$$
with $\Gamma$ piecewise regular contour describing 
$$\{\lambda \in \C : |\lambda| \geq R, |Arg(\lambda)| = \theta\} \cup \{\lambda \in \C : |\lambda| = R, |Arg(\lambda)| \leq  \theta\},$$
oriented from $\infty e^{-i\theta}$ to $\infty e^{i\theta}$, with $\theta \in (\frac{\pi}{2}, \frac{\phi}{\alpha} )$, $R \in \R^+$ such that $\lambda^\alpha \in \rho(A)$ if $\lambda \in \Gamma$. Then we can apply Proposition \ref{pr7A} to $u$, replacing 
$\alpha$ with $\alpha-1$, and observing that $g_0 \in (X, D(A))_{\frac{\beta +1}{\alpha}, \infty}$ implies $g_0 \in \overline{D(A)}$, so that in $\{\lambda \in \C : |\lambda| \geq R,  |Arg(\lambda)| \leq \theta\}$
$$
{\displaystyle \lim_{|\lambda| \to \infty}} \lambda^\alpha (\lambda^\alpha - A)^{-1} g_0 = g_0. 
$$
We deduce that $u \in D(B^{\alpha -1})$, $B^{\alpha -1}u(0) = g_0$ and, if $t \in (0, T]$, 
$$
B^{\alpha -1}u(t) = \frac{1}{2\pi i} \int_\Gamma e^{\lambda t} \lambda^{\alpha -1} (\lambda^\alpha - A)^{-1} g_0 d\lambda
$$
$$
(B^{\alpha -1}u)'(t) = \frac{1}{2\pi i} \int_\Gamma e^{\lambda t} \lambda^{\alpha} (\lambda^\alpha - A)^{-1} g_0 d\lambda = \frac{1}{2\pi i} \int_\Gamma e^{\lambda t} A(\lambda^\alpha - A)^{-1} g_0 d\lambda = Au(t). 
$$
We observe that, as $g_0 \in (X, D(A))_{\frac{\beta +1}{\alpha}, \infty}$, 
$$
\|A(\lambda^\alpha - A)^{-1} g_0\| \leq const |\lambda|^{-\beta - 1}. 
$$
Another application of Proposition \ref{pr7A} guarantees that $B^{\alpha-1} u \in C^1([0, T]; X)$ and $u \in C([0, T]; D(A))$ and $u$ is really a solution of (\ref{eq37}) if $f \equiv 0$. It remains only to show that $(B^{\alpha-1} u)'$ and $Au$  belong 
to $C^\beta([0, T]; X)$. We follow an argument already used: if $t \in (0, T]$, 
$$
(Au)'(t) = \frac{1}{2\pi i} \int_\Gamma e^{\lambda t} \lambda A(\lambda^\alpha - A)^{-1} g_0 d\lambda.
$$
The estimate
$$
\|\lambda A(\lambda^\alpha - A)^{-1} g_0 \| \leq const |\lambda|^{-\beta}
$$
implies 
$$
\|(Au)'(t) \| \leq const \cdot t^{\beta -1}, 
$$
implying $Au \in C^\beta([0, T]; X)$. 

\medskip

{\it ($\gamma_4$) Proof of the general statement.}

\medskip

The case $g_0 = 0$ or $\beta > \alpha - 1$ follows from $(\gamma_1)$. Let us consider the case $\beta < \alpha - 1$. The uniqueness of a solution with the stated regularity again follows from $(\gamma_1)$. Concerning the existence, the necessity of conditions (I)-(II) follows from $(\gamma_2)$. Suppose that $f(0) \in (X, D(A))_{\frac{\beta}{\alpha},\infty}$ and $g_0 \in (X, D(A))_{\frac{\beta+1}{\alpha},\infty}$. By $(\gamma_3)$, in case $f \equiv 0$, a solution $u_1$ with the stated regularity exists. Take as new unknown $u_2:= u - u_1$. This should solve (\ref{eq37}) with $g_0 = 0$. Another application of $(\gamma_1)$ shows the existence and proper regularity of $u_2$. 

\end{proof}

Now we introduce the {\bf Caputo's derivative of order $\alpha$} $C^\alpha$, again in case $\alpha \in \R^+ \setminus \N$. Let $m = [\alpha]$, so that $m < \alpha < m+1$. If $u \in C^{m+1}([0, T]; X)$, $C^\alpha u$ is defined as
\begin{equation}\label{eq27}
\begin{array}{c}
C^\alpha u(t) := \frac{1}{\Gamma (m+1-\alpha)} \int_0^t (t-s)^{m-\alpha}u^{(m+1)}(s) ds \\ \\
=  [B^{\alpha - (m+1)} u^{(m+1)}] (t) = B^{\alpha - (m+1)}(u - \sum_{k=0}^m \frac{t^k}{k!} u^{(k)}(0))^{(m+1)} \\ \\
= B^\alpha (u - \sum_{k=0}^m \frac{t^k}{k!} u^{(k)}(0)), 
\end{array}
\end{equation}
as $u - \sum_{k=0}^m \frac{t^k}{k!} u^{(k)}(0) \in D(B^{m+1})$. This suggests the following 

\begin{definition}\label{de22}
Let $\alpha \in \R^+ \setminus \N$, $m = [\alpha]$, $u \in C^m([0, T]; X)$. We shall say that $u \in D(C^\alpha)$ if $u -  \sum_{k=0}^m \frac{t^k}{k!} u^{(k)}(0) \in D(B^\alpha)$. In this case, we set 
$$C^\alpha u := B^\alpha (u -  \sum_{k=0}^m \frac{t^k}{k!} u^{(k)}(0)). $$
\end{definition}

We want to extend  the results concerning equation (\ref{eq20}) to the system

\begin{equation}\label{eq28}
\left\{\begin{array}{ll}
C^\alpha u(t) - Au(t) = f(t), & t \in [0, T], \\ \\
u^{(k)}(0) = u_k, \quad 0 \leq k \leq [\alpha]. 
\end{array}
\right.
\end{equation}

Of course, a {\bf strict solution } to (\ref{eq28}) is an element of $D(C^\alpha) \cap C([0, T]; D(A))$ satisfying pointwise all conditions in (\ref{eq28}). 

The case $\alpha \in (0, 1)$ is quite simple: 

\begin{proposition}\label{pr23}\label{pr23}
Let $\alpha \in (0, 1)$. Consider system (\ref{eq28}),supposing that $\lambda_0 -A$ is of type $\phi$, for some  $\phi \in (0, (1 - \frac{\alpha}{2})\pi)$, $\lambda_0 \geq 0$; then:

(I) (\ref{eq28}) has, at most, one strict solution, for every $f \in C([0, T]; X)$, $u_0 \in D(A)$. 

(II) Let $\theta \in (0, 1)$. Then necessary and sufficient conditions implyng that (\ref{eq28}) has a strict solution $u$ such that $C^\alpha u$ and $Au$ are bounded with values in $(X, D(A))_{\theta,\infty}$ are : $u_0 \in D(A)$, $Au_0 \in (X, D(A))_{\theta,\infty}$, $f \in C([0, T]; X) \cap B([0, T]; (X, D(A))_{\theta,\infty})$. 

(III)  Let $\beta \in (0, \alpha)$. Then the following conditions are necessary and sufficient, in order that (\ref{eq28}) has a strict solution $u$ such that $C^\alpha u$, $Au$ belong to $C^\beta([0, T]; X)$: $u_0 \in D(A)$, $Au_0 + f(0) \in (X, D(A))_{\beta/\alpha, \infty}$, $f \in C^\beta([0, T]; X)$. 
\end{proposition}

\begin{proof} (I) Suppose that $u_0 = 0$, $f \equiv 0$. Then (\ref{eq28}) is equivalent to 
$$
\begin{array}{ll}
B^\alpha u(t) - Au(t) = 0, & t \in [0, T],
\end{array}
$$
which, by Proposition \ref{pr15}(I), has only the trivial solution. 

(II) The necessity of these three conditions is clear. In order to show that they are sufficient, we have only to observe that $v(t):= u(t) - u_0$ should belong to $D(B^\alpha) \cap C([0, T]; D(A))$ and $B^\alpha v = C^\alpha u$ and $Av = Au - Au_0$ should 
be bounded with values in $(X, D(A))_{\theta,\infty}$. Moreover, $v$ should solve the equation
\begin{equation}\label{eq29}
B^\alpha v(t) - Av(t) = f(t) + Au_0,  \quad t \in [0, T]. 
\end{equation}
By Proposition \ref{pr15} (III), this implies that $f \in B([0, T]; (X, D(A))_{\theta,\infty})$. On the other hand, if the three conditions are satisfied, (\ref{eq29}) has a unique solution $v$ belonging to $D(B^\alpha) \cap C([0, T]; D(A))$, with $B^\alpha v, Av$ bounded with values in $(X, D(A))_{\theta,\infty}$. Clearly, if $u(t):= v(t) + u_0$, $u$ is a solution to (\ref{eq28}) with the desired properties. 

(III) Suppose a solution $u$ with the declared regularity exists. Clearly, it is necessary that $f \in C^\beta([0, T]; X)$ and $u_0 \in D(A)$. Set $v(t):= u(t) - u_0$. Then $v$ belongs to $D(B^\alpha) \cap C([0, T]; D(A))$, $B^\alpha v = C^\alpha u$ and $Av = Au - Au_0$ belong to $C^\beta([0, T]; X)$, and $v$ solves (\ref{eq29}). By Proposition \ref{pr19}, $Au_0 + f(0)$ belongs to $(X, D(A))_{\beta/\alpha, \infty}$. On the other hand, if these three conditions are satisfied, (\ref{eq29}) has a unique solution $v$ 
with $B^\alpha v$, $Av$ in $C^\beta([0, T]; X)$. It follows that $u(t):= v(t) + u_0$ is a solution to (\ref{eq28}) with the desired properties. 
\end{proof}

Now we consider the case $\alpha \in (1, 2)$. We begin with the following

\begin{lemma}\label{le27}
Let $\alpha \in (1, 2)$, let $Y$ be a Banach space, $A$ a linear operator in $Y$ satisfying (H) in Proposition \ref{pr20}. Let $\Gamma$ be a piecewise regular contour describing 
$$\{\lambda \in \C : |\lambda| \geq R, |Arg(\lambda)| = \theta\} \cup \{\lambda \in \C : |\lambda| = R, |Arg(\lambda)| \leq  \theta\},$$
oriented from $\infty e^{-i\theta}$ to $\infty e^{i\theta}$, with $\theta \in (\frac{\pi}{2}, \frac{\phi}{\alpha} )$, $R \in \R^+$ such that $\lambda^\alpha \in \rho(A)$ if $\lambda \in \Gamma$. We set, for $t \in (0, T]$, 
\begin{equation}\label{eq42A}
u(t) = \frac{1}{2\pi i} \int_\Gamma e^{\lambda t} \lambda^{\alpha -2} (\lambda^\alpha - A)^{-1} u_1 d\lambda. 
\end{equation}
Suppose that $u_1$ belongs to the closure of $D(A)$ in the interpolation space $(Y, D(A))_{1- \frac{1}{\alpha}, \infty}$. 

Then $u$ is a strict solution to (\ref{eq28}) if $f \equiv 0$, $u_0 = 0$. 

\end{lemma}

\begin{proof} We have 
$$
\|\lambda^{\alpha -2} (\lambda^\alpha - A)^{-1} u_1\|_Y \leq const |\lambda|^{-2}
$$
in $\{\lambda \in \C : |\lambda| \geq R, |Arg(\lambda)| \leq \theta\}$. Moreover, as $u_1 \in (Y, D(A))_{1- \frac{1}{\alpha}, \infty} \subseteq \overline{D(A)}$, we have (in the same set)
$$
{\displaystyle \lim_{|\lambda| \to \infty}} \lambda^2  (\lambda^{\alpha -2} (\lambda^\alpha - A)^{-1} u_1) = {\displaystyle \lim_{|\lambda| \to \infty}}   \lambda^{\alpha} (\lambda^\alpha - A)^{-1} u_1 = u_1. 
$$
We deduce from Proposition \ref{pr7A} that $u \in D(B)$, and so $u(0) = 0$, $u'(0) = Bu(0) = u_1$. Moreover, 
$$
u(t) - tu'(0) = u(t) - tu_1 = \frac{1}{2\pi i} \int_\Gamma e^{\lambda t} \lambda^{-2} [\lambda^\alpha (\lambda^\alpha - A)^{-1} - 1] u_1 d\lambda = \frac{1}{2\pi i} \int_\Gamma e^{\lambda t} \lambda^{-2} A (\lambda^\alpha - A)^{-1} u_1 d\lambda. 
$$
We have
$$
\|\lambda^{-2} A (\lambda^\alpha - A)^{-1} u_1\| \leq const |\lambda|^{-1 - \alpha}.
$$
If $u_1 \in D(A)$, 
$$
\|\lambda^{1+\alpha}(\lambda^{-2} A (\lambda^\alpha - A)^{-1} u_1)\| \leq const |\lambda|^{-1} \to 0 \quad (|\lambda| \to \infty). 
$$
We conclude that 
$$
{\displaystyle \lim_{|\lambda| \to \infty}} \lambda^{1+\alpha}(\lambda^{-2} A (\lambda^\alpha - A)^{-1} u_1) = 0. 
$$
So, by Proposition \ref{pr7A}, $u  \in D(C^\alpha)$ and, if $t \in (0, T]$,
$$
C^\alpha u(t) = \frac{1}{2\pi i} \int_\Gamma e^{\lambda t} \lambda^{\alpha -2} A (\lambda^\alpha - A)^{-1} u_1 d\lambda = Au(t). 
$$
Of course, the previous arguments imply also that $C^\alpha u(0) = 0 = Au(0)$.

\end{proof}

\begin{remark}\label{re28}
{\rm With reference to Lemma \ref{le27}, in case $u_1 \in (Y, D(A))_{1- \frac{1}{\alpha}, \infty}$, we can only say that the function $t \to \frac{1}{2\pi i} \int_\Gamma e^{\lambda t} \lambda^{\alpha -2} A (\lambda^\alpha - A)^{-1} u_1 d\lambda$ is bounded with values in $Y$. 
By the way, it is known that, if $u_1 \in (Y, D(A))_{\theta, \infty}$ for some $\theta > 1- \frac{1}{\alpha}$, it belongs to the closure of $D(A)$ in $(Y, D(A))_{1- \frac{1}{\alpha}, \infty}$. 

}
\end{remark}

\begin{theorem} 
Let $\alpha \in (1, 2)$, $A$ a linear operator in $X$ satisfying (H) in Proposition \ref{pr20}. Let $\theta \in (0, 1) \setminus \{\frac{1}{\alpha}\}$. Consider system (\ref{eq28}). Then the following conditions are necessary and sufficient, in order that there exists a unique strict solution 
$u$ such that $C^\alpha u$, $Au$ are bounded with values in $(X, D(A))_{\theta,\infty}$: 

(a) $f \in C([0, T]; X) \cap B([0, T]; (X, D(A))_{\theta,\infty})$; 

(b) $u_0 \in D(A)$, $Au_0 \in (X, D(A))_{\theta,\infty}$; 

(c) if $\theta < \frac{1}{\alpha}$, $u_1$ belongs to the interpolation space $(X, D(A))_{\theta+1 - \frac{1}{\alpha}}$; if $\frac{1}{\alpha} < \theta < 1$, $u_1 \in D(A)$ and $Au_1 \in (X, D(A))_{\theta- \frac{1}{\alpha},\infty}$. 
\end{theorem}

\begin{proof}
We begin by showing that the conditions (a)-(c) are necessary. The necessity of (a)-(b) is clear. We show the necessity of (c). Let $u$ be a strict solution to (\ref{eq28}) with the required properties. We have $u_1 = u'(0)$. Then $u \in C^1([0, T]; X) \cap C([0, T]; D(A))$, $v:= u - u(0) - tu_1 \in D(B^\alpha)$ and $B^\alpha v$, $Au$ are bounded with values in $(X, D(A))_{\theta, \infty}$. We observe that, in force of (a), $u - u(0)$ enjoys the same properties and has the same derivative. So it is not restrictive to assume $u(0) = 0$. 

We shall employ the following fact (see .... ), which is a consequence of the Fatou property of the real interpolation method:

\medskip
{\it $(\gamma_1)$  Let $-\infty \leq a < b \leq \infty$, $f \in L^1(a, b; X)$. Suppose that there exists $g \in L^1(a, b)$ such that 
$$\|f(t)\|_{(X, D(A))_{\theta,\infty}} \leq g(t) \quad (a. e. \quad in \quad (a, b)). $$
Then $\int_a^b f(t) dt \in (X, D(A))_{\theta,\infty}$ and 
$$
\|\int_a^b f(t) dt\|_{(X, D(A))_{\theta,\infty}} \leq \int_a^b g(t) dt. 
$$
 }
\medskip
Let $h(t):= B^\alpha v(t)$. Then 
$$
v(t) = \frac{1}{\Gamma (\alpha)} \int_0^t (t-s)^{\alpha -1} h(s) ds. 
$$
Let $M \in \R^+$ be such that 
$$\|h(t)\|_{(X, D(A))_{\theta,\infty}} \leq M \quad  (a. e. \quad in \quad [a, b]). $$
Then from $(\gamma_1)$ we deduce that,  $\forall t \in (0, T]$, $v(t) \in (X, D(A))_{\theta,\infty}$ and 
$$
\|v(t)\|_{(X, D(A))_{\theta,\infty}} \leq C_1 t^\alpha. 
$$
Moreover, $\forall t \in (0, T]$ 
$$
u_1 = \frac{u(t)}{t} - \frac{v(t)}{t},
$$
implying $u_1 \in (X, D(A))_{\theta,\infty}$. Now, we set, for $\xi \in (0, 1)$, 
$$X_\xi:= (X, D(A))_{\xi,\infty}, \quad \|x\|_\xi := \|x\|_{(X, D(A))_{\xi,\infty}},$$
$$
X_{1+\xi} := \{x \in D(A) : Ax \in X_\xi\}, \quad \|x\|_{1+\xi} := \max\{\|x\|, \|Ax\|_{\xi}\}. 
$$
$X_{1+\xi}$ is a Banach space with the norm $\|\cdot\|_{1+\xi}$. It is known (as a consequence of the reiteration theorem) that, if $\xi \in (0, 1) \setminus \{1-\theta\}$, 
\begin{equation}\label{eq42}
(X_\theta, X_{1+\theta})_{\xi, \infty} = X_{\theta + \xi},
\end{equation}
with equivalent norms. 

Now we set, for $t > 0$, 
$$
w(t):= \left\{\begin{array}{lll}
\frac{u(t)}{t} & {\rm if } & t \in (0, T], \\ \\
0 & {\rm if } & t \in (T, \infty). 
\end{array}
\right. 
$$
We deduce 
$$
\|u_1 - w(t)\|_\theta = \|\frac{v(t)}{t}\|_\theta \leq C_1 t^{\alpha -1}. 
$$
Moreover, 
$$
\|w(t)\|_{1+\theta} \leq C_2 t^{-1}. 
$$
If $x \in X_\theta$ and $s \in \R^+$, we set
$$
K'(s,x) := \inf \{ \|x - y\|_\theta + s\|y\|_{1+\theta} : y \in X_{1+\theta}\}
$$
Setting, for $s \in \R^+$, 
$$
z(s) := w(s^{1/\alpha}), 
$$
we deduce
$$
K'(s, u_1) \leq \|u_1 - z(s)\|_{X_\theta} + s \|z(s)\|_{X_{1+\theta}} \leq C_2 s^{1-1/2}. 
$$
So $u_1$ should belong to $(X_\theta, X_{1+\theta})_{1- 1/\alpha, \infty}$, which, together with (\ref{eq42}), implies the necessity of (c). 

Now we prove the uniqueness of a solution. Of course, it suffices to show that the only solution with all data equal to zero is the trivial one. In fact, if $u$ solves (\ref{eq28}) with $f \equiv 0$, $u_0 = u_1 = 0$, then $u$ solves (\ref{eq20}) with $F \equiv 0$, and so the conclusion follows from Proposition \ref{pr15}. 

It remains to show the existence, in case (a)-(c) hold. By linearity, we can show that a solution with the stated regularity exists in each in the three cases where two of the three data are zero. The case $u_0 = u_1 = 0$ is covered by Proposition \ref{pr15} (III). 
The case $u_1 = 0$ can be treated following the same argument of the case $\alpha \in (0, 1)$ (see the proof of Proposition \ref{pr23}, (II)). It remains to consider the case $u_0 = 0$, $f \equiv 0$ and $u_1$ as in (c). For this, we can employ Lemma \ref{le27}.
In fact, formula (\ref{eq42A}) furnishes a strict solution (recall Remark \ref{re28}). We have only to show that this strict solution $u$ is such that $Au$ (and so also $C^\alpha u$) are bounded with values in $X_\theta$. If $t \in (0, T]$, we have
$$
Au(t) = \frac{1}{2\pi i} \int_\Gamma \lambda^{\alpha -2} A (\lambda^\alpha - A)^{-1} u_1 d\lambda. 
$$
The part of $A$ in $X_\theta$, with domain $X_{1+\theta}$, satisfies (H) in Proposition \ref{pr20}. Moreover, $u_1 \in (X_\theta, X_{1+\theta})_{1 - 1/\alpha, \infty}$. So we can employ Lemma \ref{le27}, together with Remark \ref{re28}, with $Y = X_\theta$, to deduce that $Au$ is bounded with values in $X_\theta$. 

\end{proof}

It remains to extend (III) in Proposition \ref{pr23}  to the case $\alpha \in (1, 2)$. 

\begin{theorem} 
Let $\alpha \in (1, 2)$, $A$ a linear operator in $X$ satisfying (H) in Proposition \ref{pr20}. Let $\beta \in (0, \alpha) \setminus \{1\}$. Consider system (\ref{eq28}). Then the following conditions are necessary and sufficient, in order that there exists a unique strict solution 
$u$ such that $C^\alpha u$, $Au$ are bounded with values in $C^\beta ([0, T]; X)$: 

(a) $f \in C^\beta ([0, T]; X)$; 

(b) $u_0 \in D(A)$; 

(c) if $\beta \in (0, 1)$, $u_1 \in (X, D(A))_{1 - \frac{1-\beta}{\alpha}, \infty}$; if $\beta \in (1, 2)$, $u_1 \in D(A)$; 

(d) $Au_0 + f(0) \in (X, D(A))_{\frac{\beta}{\alpha},\infty}$; 

(e) if $\beta \in (1, 2)$, $Au_1 + f'(0) \in (X, D(A))_{\frac{\beta-1}{\alpha},\infty}$. 

\end{theorem}

\begin{proof} We show the claim in several steps. 

\medskip
{\it ($\gamma_1$) The fact that (a)-(b) and (c) in case $\beta > 1$ are necessary is clear. }
\medskip

{\it ($\gamma_2$) We show that (c) is necessary  in case $\beta \in (0, 1)$. }
\medskip

Let $u$ be a solution with the declared regularity. Then, by Proposition \ref{pr10}, 
$$
u(t) = u_0 + t u_1 + t^\alpha f_0 + v(t),
$$
with $f_0 \in X$ and $v \in$  $\stackrel{o}{\mathcal C}^{\alpha+\beta}([0, T]; X)$. We deduce, for $t \in (0, \frac{T}{2}]$, 
$$
u_1 = \frac{2^\alpha [u(t) - u_0] - [u(2t) - u_0]}{2^{\alpha -1}t} + \frac{v(2t) - 2^\alpha v(t)}{2^{\alpha -1}t}. 
$$
Setting
$$
w(t):= \frac{2^\alpha [u(t) - u_0] - [u(2t) - u_0]}{2^{\alpha -1}t}, \quad t \in (0, \frac{T}{2}], 
$$
We obtain 
$$
\|u_1 - w(t)\| = \|\frac{v(2t) - 2^\alpha v(t)}{2^{\alpha -1}t}\| \leq C_1 t^{\alpha +\beta -1}
$$
and 
$$
\|w(t)\|_{D(A)} \leq C_2 t^{\beta -1}. 
$$
which implies $(\gamma_2)$. 

\medskip

{\it ($\gamma_3$) Let  $\beta \in (0, 1)$, $f \equiv 0$, $u_0 = 0$, $u_1 \in  (X, D(A))_{1 - \frac{1-\beta}{\alpha}, \infty}$; then (\ref{eq28}) has a solution $u$ such that $C^\alpha u$, $Au \in C^\beta([0, T]; X)$. }
\medskip

\medskip

We consider the function $u$ defined in (\ref{eq42A}). By Lemma \ref{le27} and Remark \ref{re28}, $u$ is a strict solution to (\ref{eq28}).  It remains to show that $Au$ (and so also $C^\alpha u$) belong to $C^\beta([0, T]; X)$. The argument is the 
usual: if $t \in (0, T]$, 
$$
Au'(t) = \frac{1}{2\pi i} \int_\Gamma \lambda^{\alpha -1} A(\lambda^\alpha - A)^{-1} u_1 d\lambda. 
$$
so that
$$
\|\lambda^{\alpha -1} A(\lambda^\alpha - A)^{-1}u_1\| \leq C_3 |\lambda|^{-\beta}, 
$$
and, with the usual arguments, 
$$
\|Au'(t)\| \leq C_4 t^{\beta - 1}
$$
and $Au \in C^\beta([0, T]; X)$. 

\medskip

{\it ($\gamma_4$) If $\beta \in (0, 1)$, Condition (d) is necessary.  }
\medskip

In fact, let $u$ be a solution to (\ref{eq28}) with the desired regularity. Let $u_1$ be the function $u$ defined in (\ref{eq42A}). We set $u_2(t):= u(t) - u_1(t)$. Then, by $(\gamma_3)$, $u_2$ is a solution with data $f$, $u_0$, $0$. We set
$$
v(t) := u_2(t) - u_0. 
$$
Then, as $u_0 \in D(A)$, $v \in D(B^\alpha) \cap C^\beta ([0, T]; D(A))$, $B^\alpha v = C^\alpha u_2 \in C^\beta ([0, T]; X)$ and $v$ solves (\ref{eq29}). By Proposition \ref{pr19}, condition (d) holds. 

\medskip

{\it ($\gamma_5$) The claim holds in case $\beta \in (0, 1)$.  }
\medskip

\medskip
We have already seen that conditions (a)-(d) are necessary. We show that they are also sufficient. The uniqueness of a strict solution follows from Propositions \ref{pr19}-\ref{pr20}. In fact, if all the data are zero, $u$ is a solution of (\ref{eq20}) with $f \equiv 0$. This implies $u(t) \equiv 0$. In order to show the existence, we subtract to $u$ the function
$$u_1(t) = \frac{1}{2\pi i} \int_\Gamma e^{\lambda t} \lambda^{\alpha -2} (\lambda^\alpha - A)^{-1} u_1 d\lambda.$$
By $(\gamma_3)$ $u_1$ is a solution with the declared regularity to (\ref{eq28}) if we replace $u_0$ and $f$ with $0$. So, by difference, $u_2:= u-u_1$ is a solution in case $u_1 = 0$. We set
$$
v(t):= u_2(t) - u_0.
$$
Then $B^\alpha v = C^\alpha u_2 \in C^\beta([0, T]; X)$ and
$$
Av(t) = Au_2(t) - Au_0 \in C^\beta([0, T]; X). 
$$
Moreover, 
\begin{equation}\label{eq44}
B^\alpha v(t) - Av(t) =  f_1(t):= f(t) + Au_0, \quad t \in [0, T]. 
\end{equation}
By Proposition \ref{pr19}, necessarily $f_1(0) = Au_0 + f(0) \in (X, D(A))_{\frac{\beta}{\alpha}, \infty}$. Suppose that $\beta \in (0, 1)$ and (a)-(d) hold. Subtracting $u_1$, we are reduced to the case $u_1 = 0$. Consider the equation (\ref{eq44}). Then (a)-(d) allow to apply Proposition \ref{pr19}, assuring that (\ref{eq44}) has a unique solution $v$ with $B^\alpha v$, $Av$ in $C^\beta([0, T]; X)$. Setting
$$
u(t):= v(t) + u_0, 
$$
it is easily seen that $u$ is a solution with the desired regularity. 

\medskip
{\it ($\gamma_6$) The claim holds in case $\beta \in (1, 2)$. }
\medskip

In this case, both $u_0$, $u_1$ necessarily belong to $D(A)$. Set
$$
v(t) := u(t) - u_0 - t u_1. 
$$
Then $v \in D(B^\alpha) \cap C^\beta([0, T]; D(A))$ and $B^\alpha v = C^\alpha u \in C^\beta([0, T]; D(A))$. Moreover, $v$ is a solution to
\begin{equation}\label{eq45}
B^\alpha v(t) - Av(t) = f_1(t) = f(t) + Au_0 + tAu_1. 
\end{equation}
By Proposition \ref{pr19}, necessarily
$$
f_1(0) = Au_0 + f(0) \in (X, D(A))_{\frac{\beta}{\alpha}, \infty}, \quad f'_1(0) = Au_1 + f'(0) \in (X, D(A))_{\frac{\beta-1}{\alpha}, \infty}.
$$
So the conditions (a)-(e) are necessary. To show that they are sufficient, we can solve (\ref{eq45}): applying Proposition \ref{pr19}, we deduce that there is a unique solution $v$ each that $B^\alpha v$, $Av$ belong to $C^\beta ([0, T]; X)$. Then 
$$
u(t) = v(t) + u_0 + tv_1
$$
is a solution to (\ref{eq28}) with the declared properties. 

\end{proof}

\begin{remark}\label{re24}
{\rm It may be of interest an explicit formula for the solution (if existing) to (\ref{eq28}). We begin by considering the case $\alpha \in (0, 1)$. We have $u(t) = v(t) + u_0$, with $v$ solving (\ref{eq29}). From Propositions \ref{pr13}-\ref{pr15}, we get 
$$
\begin{array}{c}
u(t)
= u_0 +  \frac{1}{2\pi i} \int_0^t (\int_\Gamma e^{\mu(t-s)} (\mu^\alpha - A)^{-1} d\mu) Au_0 ds + \delta(\alpha) \frac{1}{2\pi i} \int_\Gamma e^{\mu t} \mu^{\alpha-2}(\mu^\alpha - A)^{-1} u_1 d\mu  \\ \\
+   \frac{1}{2\pi i} \int_0^t (\int_\Gamma e^{\mu(t-s)} (\mu^\alpha - A)^{-1} d\mu) f(s) ds, 
\end{array}
$$
with 
$$
\delta(\alpha) = \left\{\begin{array}{lll}
0 & {\rm if } & 0 < \alpha < 1, \\ \\
1 & {\rm if } & 1 < \alpha < 2, 
\end{array}
\right. 
$$
 $\Gamma$ joining $\infty e^{-i\theta}$ to $\infty e^{i\theta}$ for some $\theta \in (\frac{\pi}{2}, \pi)$, and $\mu^\alpha \in \rho(A)$ $\forall \lambda \in \Gamma$ (see Remark \ref{re21}).  We set, for $t \in \R^+$, 
\begin{equation}
\begin{array}{c}
H(t) := \frac{1}{2\pi i} \int_\Gamma e^{\mu t} \mu^{\alpha -1}(\mu^\alpha - A)^{-1} d\mu, \\ \\
\end{array}
\end{equation}
\begin{equation}
\begin{array}{c}
S(t) := \frac{1}{2\pi i} \int_\Gamma e^{\mu t} (\mu^\alpha - A)^{-1} d\mu. 
\end{array}
\end{equation}
From  We deduce that $S(t) \in \mL(X)$ and, if $t \in (0, 1]$, 
$$
\|S(t)\|_{\mL(X)} \leq C_1 \int_{t^{-1} \Gamma_1} e^{tRe(\mu)} |\mu|^{-\alpha} |d\mu| = C_1 t^{\alpha -1} \int_{\Gamma_1} e^{Re(\mu)} |\mu|^{-\alpha} |d\mu|.
$$
We deduce that
$$
\|S(t)\|_{\mL(X)} \leq C t^{\alpha - 1}, \quad \forall t \in (0, T]. 
$$
From Proposition \ref{pr7A}, we deduce also that
$$
\|H(t)\|_{\mL(X)} \leq C,
$$
and, in case $f \in \overline{D(A)}$, 
$$
\lim_{t \to 0} H(t)f = f. 
$$
We have
$$
\begin{array}{c}
\int_0^t S(s) Au_0 ds \\ \\
= \frac{1}{2\pi i} \int_\Gamma \frac{e^{\mu t} - 1}{\mu} (\mu^\alpha - A)^{-1} Au_0 d\mu = \frac{1}{2\pi i} \int_\Gamma e^{\mu t} \mu^{-1} A(\mu^\alpha - A)^{-1} u_0 d\mu \\ \\
= -\frac{1}{2\pi i} \int_\Gamma e^{\mu t} \mu^{-1} u_0 d\mu + H(t) u_0  \\ \\
= - u_0 + H(t) u_0,
\end{array}
$$
as $\int_\Gamma \mu^{-1} (\mu^\alpha - A)^{-1} Au_0 d\mu = 0$. 
So 
$$
u_0 +  \int_0^t S(s) Au_0 ds = H(t)u_0. 
$$
Let $\gamma$ be a piecewise, simple, regular path, joining $\infty e^{-i\theta_1}$ to $\infty e^{i\theta_1}$,  for some $\theta_1 \in (\theta, \pi)$, contained in $\C \setminus (-\infty, 0]$ and in same connected component of $\C \setminus \Gamma$ containing $0$. Then, by Proposition \ref{pr2},  for $t \in (0, T]$, 
$$
\begin{array}{c}
\frac{1}{\Gamma (1-\alpha)} \int_0^. (.-s)^{-\alpha} S(s) u_0 ds = - \frac{1}{2\pi i} \int_\gamma \lambda^{\alpha - 1} (\lambda - B)^{-1} (\frac{1}{2\pi i} \int_\Gamma e^{\mu .} (\mu^\alpha - A)^{-1}u_0 d\mu) d\lambda \\ \\
= - \frac{1}{2\pi i} \int_\gamma \lambda^{\alpha - 1}  (\frac{1}{2\pi i} \int_\Gamma e^{\mu .} (\lambda - \mu)^{-1}  (\mu^\alpha - A)^{-1}u_0 d\mu) d\lambda \\ \\
=  \frac{1}{2\pi i} \int_\Gamma (-\frac{1}{2\pi i} \int_\gamma  \lambda^{\alpha - 1}   (\lambda - \mu)^{-1} d\lambda)  e^{\mu .} (\mu^\alpha - A)^{-1} u_0 d\mu \\ \\
= \frac{1}{2\pi i} \int_\Gamma e^{\mu .} \mu^{\alpha - 1} (\mu^\alpha - A)^{-1}u_0 d\mu
\end{array}
$$
Moreover, 
$$
\frac{1}{2\pi i} \int_\Gamma e^{\mu t} \mu^{\alpha-2}(\mu^\alpha - A)^{-1} u_1 d\mu = \int_0^t H(s) u_1 ds. 
$$
We deduce the following
\begin{equation}\label{eq31}
u(t) = H(t)u_0 + \delta(\alpha) \int_0^t H(s) u_1 ds + \int_0^t S(t-s) f(s) ds.
\end{equation}

}
\end{remark}

\begin{example}
{\rm Let $\Omega$ be an open, bounded subset of $\R^n$, lying on one side of its boundary $\partial \Omega$, which is an submanifold of $\R^{n}$ of dimension $n-1$ and class $C^3$. We introduce in $\Omega$ the following operator $\tilde A$: 
\begin{equation}\label{eq32}
\left\{ \begin{array}{l}
D(A) = \{u \in \cap_{1 \leq p < \infty} W^{2,p}(\Omega) : \Delta u \in C(\overline \Omega), D_\nu u_{|\partial \Omega} = 0\}, \\ \\
A u = \Delta u. 
\end{array}
\right. 
\end{equation}
We think of $A$ as an unbounded operator in 
\begin{equation}
X:= C(\overline \Omega). 
\end{equation}
The spectrum $\sigma(A)$ of $A$ is contained in $(-\infty, 0]$. Moreover, $\forall \epsilon \in (0, \pi)$, $\forall \lambda_0 \in \R^+$ there exists $C(\epsilon)$ such that if $\lambda \in \C$ and $|\lambda| \geq \epsilon$ and $|Arg(\lambda)| \leq \pi - \epsilon$, 
$$
\|(\lambda - A)^{-1}\|_{\mL(X)} \leq C(\epsilon) |\lambda|^{-1}
$$
Now we set
\begin{equation}
A:= e^{i\phi} A, 
\end{equation}
with $\phi \in (-\pi, \pi]$. 
Then we can say that, $\forall \lambda_0 \in \R^+$, $A + \lambda_0$ is an operator of type $\frac{\pi}{2}$. 
Moreover, it is known (see \cite{Gu1}) that, $\forall \theta \in (0, 1) \setminus \{0\}$, 
\begin{equation}\label{eq36}
(C(\overline \Omega), D(A))_{\theta,\infty} = \left\{\begin{array}{lll}
C^{2\theta}(\overline \Omega) & {\rm if } & 0 < \theta < \frac{1}{2}, \\ \\
\{u \in C^{2\theta}(\overline \Omega): D_\nu u_{|\partial \Omega} = 0\}  & {\rm if } & \frac{1}{2} < \theta < 1. 
\end{array}
\right. 
\end{equation}
We fix $\alpha$ in $(0, 2) \setminus \{1\}$, $\phi \in (-\pi, \pi]$  and consider the mixed problem

\begin{equation}\label{eq35}
\left\{\begin{array}{ll}
C^\alpha u(t,x) -  e^{i\phi} \Delta_x u(t,x) = f(t,x), & (t,x) \in [0, T] \times \Omega, \\ \\
D_{\nu} u(t,x') = 0, & (t,x') \in [0, T] \times \partial \Omega, \\ \\
D_t^k u(0,x) = u_k(x), & k \in \N_0, k < \alpha,  x \in \Omega. 
\end{array}
\right. 
\end{equation}
Let $\beta, \gamma \in [0, \infty)$, $f : [0, T] \times \overline \Omega \to \C$. We shall write $f \in C^{\beta,\gamma}([0, T] \times \overline \Omega)$ if $f \in C^\beta ([0, T]; C(\overline \Omega)) \cap B([0, T]; C^\gamma(\overline \Omega))$. Then the following result holds: 

\begin{theorem}\label{th26}
Consider system (\ref{eq35}). Let $\alpha \in (0, 2) \setminus \{1\}$, $|\phi| < \frac{(2-\alpha) \pi}{2}$. Then: 

(I) if $\gamma \in (0, 1)$, $\gamma - \frac{2}{\alpha} \not \in \Z$,  the following conditions are necessary and sufficient, in order that there exist a unique strict solution $u$ belonging to $D(C^\alpha) \cap C([0, T]; D(A))$, with $C^\alpha u$ and $Au$ bounded with values in $C^\gamma (\overline \Omega)$: 

(a) $u_0 \in C^{2+\gamma}(\overline \Omega)$, $D_\nu u_0 = 0$; 

(b) $f \in C([0, T] \times \overline \Omega) \cap B([0, T]; C^{\gamma}(\overline \Omega)$; 

(c) if $\alpha \in (1, 2)$, $u_1 \in C^{\gamma + 2 - \frac{2}{\alpha}}(\overline \Omega)$ and, if $\gamma + 2 - \frac{2}{\alpha} > 1$, $D_\nu u_1 = 0$. 

Let $\beta \in (0, \alpha) \setminus \{1\}$ $\frac{2\beta}{\alpha}  \not \in \Z$. Then the following conditions are necessary and sufficient, in order that there exist a unique strict solution $u$  such that $C^\alpha u$ and $Au = \Delta u$ belong to
$C^{\beta}([0, T]; C(\overline \Omega))$: 

(d) $u_0 \in D(A)$, $e^{i\phi} \Delta u_0 + f(0,\cdot) \in C^{\frac{2\beta}{\alpha}}(\overline \Omega)$ and, if $\frac{2\beta}{\alpha} > 1$, $D_\nu(e^{i\phi} \Delta u_0 + f(0,\cdot)) = 0$; 

(e) $f \in C^\beta([0, T]; C(\overline \Omega))$; 

(f) if $\alpha \in (1, 2)$, $\beta \in (0, 1)$, $u_1 \in C^{2- \frac{2(1-\beta)}{\alpha}}(\overline \Omega)$ and, if $2- \frac{2(1-\beta)}{\alpha} > 1$, $D_\nu u_1 = 0$; if $1 < \beta < \alpha < 2$, $u_1 \in D(A)$; 

(g) if $1 < \beta < \alpha < 1$, $e^{i\phi} \Delta u_1 + D_t f(0,\cdot) \in C^{\frac{2(\beta - 1)}{\alpha}}(\overline \Omega)$. 
\end{theorem}

\begin{proof}
We employ Proposition (\ref{pr23}). We observe that $\frac{\alpha \pi}{2} + \frac{\pi}{2} < \pi$ $\forall \alpha \in (0, 1)$. So the proposition is applicable. The conclusion easily follows from (\ref{eq36}), taking $\theta = \frac{\gamma}{2}$ and recalling that $u_0 \in D(A)$ and $\Delta u_0 \in C^\gamma (\overline \Omega)$ imply $u_0 \in C^{2+\gamma}(\overline \Omega)$. 
\end{proof}

The "natural" relationship between $\alpha$, $\beta$, $\gamma$ is $\frac{2\beta}{\alpha} = \gamma$ (recall the case $\alpha = 1$ for parabolic problems, see \cite{Lu1}). 
}
\end{example}


\begin{thebibliography}{99}

\bibitem{DaGr1} G. Da Prato, P. Grisvard, "Sommes d'opérateurs linéaires et equations différentielles opérationelles", J. Math. Pures Appliquees {\bf 54} (1975), 305-387. 

\bibitem{Gu1} D. Guidetti, "On interpolation with boundary conditions",  Math. Z., {\bf 207} (1991), 439--460.

\bibitem{Gu2} D. Guidetti, "An introduction to maximal regularty for parabolic problems and interpolation theory with application to an inverse problem", in "Interplay between $C_0-$semigroups and PDEs theory and applications" (ed. S. Romanelli, R.M. Mininni, S. Lucente), Istituto Nazionale di Alta Matematica (2004). 

\bibitem{Lu1} A. Lunardi: {\it Analytic semigroups and optimal regularity in parabolic problems,}
Birk\-h\"au\-ser 1995.


\bibitem{Po1} I. Podlubny, {\it Fractional Differential Equations}, Mathematics in Science and Engineering vol. 198, Academic Press (1999). 


\bibitem{Ta1} H. Tanabe, {\it Equations of Evolution}, Pitman (1979). 

\bibitem{Tr1} H. Triebel {\it Interpolation Theory, Function Spaces, Differential Operators}, North Holland Mathematical Library, vol. 18 (1978). 


\end{thebibliography}
\end{document}